\documentclass[10pt]{article}
\usepackage[a4paper]{geometry}

\usepackage[T1]{fontenc}
\usepackage[utf8]{inputenc}
\usepackage[british]{babel}
\usepackage[dvipsnames,svgnames,table]{xcolor}
\usepackage{lmodern}
\usepackage{csquotes}
\usepackage{amsmath,amsthm,amsfonts,amssymb,mathtools}
\usepackage{enumitem}
\usepackage{thmtools}
\usepackage{commath}
\usepackage{nicefrac}
\usepackage{stmaryrd}
\usepackage{etoolbox}
\usepackage{booktabs}
\usepackage{multirow}
\usepackage{hyperref}
\usepackage[capitalize]{cleveref}
\usepackage{graphicx}
\usepackage{subcaption}
\usepackage{todonotes}
\usepackage{tabularray}
\usepackage{tikz}
\usepackage{cite}

\usepackage{tikz-3dplot}
\usetikzlibrary{positioning, shapes, arrows, calc, arrows.meta, patterns, shadows, trees, intersections}

\DeclareMathOperator*{\argmin}{arg\,min}

\newtheorem{theorem}{Theorem}[section] %
\newtheorem{lemma}[theorem]{Lemma}

\newtheorem{remark}[theorem]{Remark}

\newtheorem{corollary}[theorem]{Corollary}
\newtheorem{assumption}[theorem]{Assumption}

\newcommand{\diam}{\mathrm{diam}}

\newcommand{\Th}{{\mathcal{T}_h}}
\newcommand{\elm}{T}                %
\newcommand{\Fh}{{\mathcal{F}_h}}
\newcommand{\Fhi}{{\mathcal{F}_h^{\text{i}}}}
\newcommand{\Fhb}{{\mathcal{F}_h^{\text{b}}}}

\newcommand{\dx}{\mathop{\mathrm{d} x}}

\newenvironment{eqs} %
 { \begin{equation} \begin{aligned} } %
 { \end{aligned} \end{equation} \ignorespacesafterend } %

\renewcommand{\norm}[1]{\lVert #1 \rVert}
\newcommand{\dgnorm}[1]{\norm{#1}_{\IX}}
\newcommand{\dgsnorm}[1]{\norm{#1}_{\IX,\ast}}
\newcommand{\vdgnorm}[1]{\norm{#1}_{1,h}}

\newcommand{\stablocal}[1]{\alpha_{\scalebox{0.6}{$K_{#1}$}}}
\newcommand{\contlocal}[1]{\beta_{\scalebox{0.6}{$K_{#1}$}}}

\newcommand{\contoseen}[1]{\beta_{\scalebox{0.6}{$k_{#1}$}}}
\newcommand{\staboseenr}[1]{\alpha_{\scalebox{0.6}{$\underline{k}_{#1}$}}}
\newcommand{\contoseenr}[1]{\beta_{\scalebox{0.6}{$\underline{k}_{#1}$}}}
\newcommand{\contcross}[1]{\beta_{\scalebox{0.6}{$\times_{#1}$}}}

\newcommand{\conttrefftz}[1]{\beta_{\scalebox{0.6}{$c_{#1}$}}}
\newcommand{\stabtrefftz}[1]{\alpha_{\scalebox{0.6}{$c_{#1}$}}}
\newcommand{\stabtrefftzr}[1]{\alpha_{\scalebox{0.6}{$\underline{c}_{#1}$}}}

\newcommand{\stabsolr}[1]{C_{\scalebox{0.6}{$\underline{S}_{#1}$}}}
\newcommand{\banach}{\Lambda_*}
\newcommand{\Cres}[1]{C_{\texttt{res},#1}}
\newcommand{\contconv}{\beta_{\scalebox{0.6}{${a}^{\texttt{c}}$}}}   %

\DeclareMathOperator{\id}{id}

\DeclareMathOperator{\Div}{\mathrm{div}}%
\DeclareMathOperator{\Id}{\mathrm{id}}%

\newcommand\restr[2]{{\left.\kern-\nulldelimiterspace #1 \right|_{#2} }}

\newcommand{\avg}[1]{\{\!\!\{ #1\}\!\!\}}
\newcommand{\jmp}[1]{[\![#1]\!]}

\newcommand{\up}{\xi}    %
\newcommand{\vq}{\eta}   %
\newcommand{\rs}{\sigma} %
\newcommand{\z}{\zeta}   %

\newcommand{\IB}{\mathbb{B}}

\newcommand{\IL}{\mathbb{L}}

\newcommand{\IP}{\mathbb{P}}
\newcommand{\IQ}{\mathbb{Q}}
\newcommand{\IR}{\mathbb{R}}

\newcommand{\IT}{\mathbb{T}}

\newcommand{\IW}{\mathbb{W}}
\newcommand{\IX}{\mathbb{X}}
\newcommand{\IXreg}{\mathbb{X}^{\texttt{reg}}}

\newcommand{\IZ}{\mathbb{Z}}

\newcommand{\TR}{\underline{\mathbb{T}}}
\newcommand{\QR}{\underline{\mathbb{Q}}}
\newcommand{\LR}{\underline{\mathbb{L}}}
\newcommand{\XR}{\underline{\mathbb{X}}}

\newcommand{\ZR}{\underline{\mathbb{Z}}}
\newcommand{\XRreg}{\underline{\mathbb{X}}^{\texttt{reg}}}
\newcommand{\ur}{\underline{u}}
\newcommand{\vr}{\underline{v}}
\newcommand{\rr}{\underline{r}}
\newcommand{\zr}{\underline{z}}
\newcommand{\Plift}[1]{\mathcal{P}_{\!#1}} %

\newcommand{\blfvis}{a_h^{\!\texttt{d}}}             %
\newcommand{\blfdiv}{b_h}                             %
\newcommand{\blfconv}[1]{a_{h,#1}^{\!\texttt{c}}}    %
\newcommand{\blfoseen}[1]{k_{h,#1}}                  %
\newcommand{\blftrefftz}[1]{c_{h,#1}}                      %
\newcommand{\blfoseenr}[1]{\underline{k}_{h,#1}}     %
\newcommand{\blftrefftzr}[1]{\underline{c}_{h,#1}}         %

\newcommand{\oposeenelem}[1]{K_{\elm,#1}}                    %
\newcommand{\oposeen}[1]{K_{#1}}                          %
\newcommand{\oposeenrelem}[1]{\underline{K}_{\elm,#1}}       %
\newcommand{\oposeenr}[1]{\underline{K}_{#1}}             %
\newcommand{\rmod}[1]{\underline{#1}}

\newcommand{\calT}{\mathcal{T}}

\newcommand*\circled[1]{\tikz[baseline=(char.base)]{
            \node[shape=circle,draw,inner sep=2pt,scale=0.8] (char) {#1};}}

\usepackage{pgfplots}
\usepgfplotslibrary{colorbrewer,groupplots}
\pgfplotsset{
  compat=1.15,
  cycle list/Set1-5,
  title style = {font = \small},
  legend style = {font = \small},
  label style = {font = \footnotesize},
  tick label style = {font = \footnotesize},
  PPk/.style={mark=square*},
  TTk/.style={mark=*},
  TTkw/.style={mark=triangle*},
  TTkemb/.style={mark=diamond*},
}
\pgfplotsset{
    discard if not/.style 2 args={
        x filter/.append code={
            \edef\tempa{\thisrow{#1}}
            \edef\tempb{#2}
            \ifx\tempa\tempb
            \else
                
            \fi
        }
    }
}
\pgfplotsset{
    discard if/.style 2 args={
        x filter/.append code={
            \edef\tempa{\thisrow{#1}}
            \edef\tempb{#2}
            \ifx\tempa\tempb
                
            \fi
        }
    }
}

\pgfplotscreateplotcyclelist{paulcolors3}{%
violet,every mark/.append style={solid,fill=violet},mark=square*,very thick,mark size=3pt\\%
teal,every mark/.append style={solid,fill=teal},mark=*,very thick,mark size=2.8pt\\%
orange,every mark/.append style={solid,fill=orange},mark=diamond*,very thick,mark size=2pt\\%
violet,densely dashed,every mark/.append style={solid,fill=violet},mark=square*,very thick,mark size=3pt\\%
teal,densely dashed,every mark/.append style={solid,fill=teal},mark=*,very thick,mark size=2.8pt\\%
orange,densely dashed,every mark/.append style={solid,fill=orange},mark=diamond*,very thick,mark size=2pt\\%
violet,dash dot,every mark/.append style={solid,fill=violet},mark=square*,very thick,mark size=3pt\\%
teal,dash dot,every mark/.append style={solid,fill=teal},mark=*,very thick,mark size=2.8pt\\%
orange,dash dot,every mark/.append style={solid,fill=orange},mark=diamond*,very thick,mark size=2pt\\%
}

\pgfplotscreateplotcyclelist{paulcolors4}{%
violet,every mark/.append style={solid,fill=violet},mark=square*,very thick,mark size=3pt\\%
teal,every mark/.append style={solid,fill=teal},mark=*,very thick,mark size=3pt\\%
orange,every mark/.append style={solid,fill=orange},mark=diamond*,very thick,mark size=3pt\\%
cyan,every mark/.append style={solid,fill=cyan},mark=star,very thick,mark size=3pt\\%
violet,densely dashed,every mark/.append style={solid,fill=violet},mark=square*,very thick,mark size=3pt\\%
teal,densely dashed,every mark/.append style={solid,fill=teal},mark=*,very thick,mark size=2.8pt\\%
orange,densely dashed,every mark/.append style={solid,fill=orange},mark=diamond*,very thick,mark size=2pt\\%
cyan,densely dashed,every mark/.append style={solid,fill=cyan},mark=star,very thick,mark size=3pt\\%
violet,dash dot,every mark/.append style={solid,fill=violet},mark=square*,very thick,mark size=3pt\\%
teal,dash dot,every mark/.append style={solid,fill=teal},mark=*,very thick,mark size=2.8pt\\%
orange,dash dot,every mark/.append style={solid,fill=orange},mark=diamond*,very thick,mark size=2pt\\%
cyan,dash dot,every mark/.append style={solid,fill=cyan},mark=star,very thick,mark size=3pt\\%
}

\pgfplotscreateplotcyclelist{paulcolors2}{%
{violet,every mark/.append style={solid,fill=violet},mark=square*,very thick,mark size=3pt}\\%
{violet,densely dashed,every mark/.append style={solid,fill=violet},mark=square*,very thick,mark size=3pt}\\%
{teal,every mark/.append style={solid,fill=teal},mark=*,very thick,mark size=3pt}\\%
{teal,densely dashed,every mark/.append style={solid,fill=teal},mark=*,very thick,mark size=2.8pt}\\%
}

\usepackage{pgfplotstable} 
\makeatletter
\pgfplotstableset{
	discard if not/.style 2 args={
	row predicate/.append code={
	\def\pgfplotstable@loc@TMPd{\pgfplotstablegetelem{##1}{#1}\of}
	\expandafter\pgfplotstable@loc@TMPd\pgfplotstablename
	\edef\tempa{\pgfplotsretval}
	\edef\tempb{#2}
	\ifx\tempa\tempb
	\else
	\pgfplotstableuserowfalse
	\fi
	}
	}
}
\makeatother

\usepackage[plain]{algorithm}
\usepackage{algpseudocodex}

\title{
Embedded Trefftz DG method for steady Navier--Stokes flow
\\[4pt]\large Part II: Nonlinear problem} 

\author{%
  Paul Stocker\thanks{Faculty of Mathematics, University of Vienna,
  Oskar-Morgenstern-Platz 1, 1090 Vienna, Austria.}
  \and
  Igor Voulis\thanks{Institute for Numerical and Applied Mathematics,
  University of G\"ottingen, Lotzestr.\ 16--18, 37083 G\"ottingen, Germany.}
  \and
  Christoph Lehrenfeld\footnotemark[2]
  \and
  Philip L.\ Lederer\thanks{Department of Mathematics, University of Hamburg,
  Bundesstra{\ss}e 55, 20146 Hamburg, Germany.}
}

\date{}

\begin{document}
\maketitle

\begin{abstract}
We develop and analyze an embedded Trefftz-DG method for the steady incompressible Navier--Stokes equations, based on the reduced Oseen discretization from Part~I.
The main difficulty is that the reduced Trefftz space depends on the convection field, so successive Picard iterates live in different discrete spaces.
We address this by constructing projections between convection-dependent Trefftz spaces and using them to control the reduced Oseen solution map.
Under suitable resolution and small-data assumptions, we prove existence of discrete solutions, uniqueness, and convergence of the Picard iteration.
We also derive an a priori error analysis by relating the method to the underlying DG discretization, thereby inheriting convergence properties from compatible DG Navier--Stokes analyses.
Numerical experiments on standard incompressible-flow benchmarks illustrate the theory.


\end{abstract}

\section{Introduction}\label{sec:intro}

This paper is the second part of a two-part study on embedded Trefftz discontinuous Galerkin methods for incompressible flow problems.
While Part~I develops the linear theory for the steady Oseen problem, the present paper addresses the steady incompressible Navier--Stokes equations.
Our main goal is to show that the embedded Trefftz framework can be extended to the nonlinear setting and to establish existence, uniqueness, and convergence results for the resulting discrete problem.
A detailed comparison with standard DG methods in terms of computational cost was given in \cite{LLS_NM_2024} for the Stokes problem; the same local elimination mechanism underlies the Navier--Stokes method considered here.

Trefftz methods, dating back to Trefftz \cite{trefftz1926}, use approximation spaces adapted to the kernel of the underlying differential operator and can thereby achieve high approximation quality with significantly fewer degrees of freedom than standard polynomial discretizations.
Nonlinear problems, variable coefficients, and nonzero right-hand sides are, however, well known to be major obstacles for classical Trefftz constructions.

The embedded Trefftz-DG approach, introduced in \cite{LS_IJMNE_2023,lozinski19}, provides a way around this obstruction for linear problems and has been applied to a range of model equations; see, for example, \cite{SV_ARXIV_2026,HLSW_M2AN_2022,M25,S23,H24,PIMPS_ARXIV_2026,GPS_JSC_2025}.
For a detailed comparison with related methods on general meshes, we refer to \cite{LSZ_PAMM_2024}.
The key idea is to replace the explicit construction of Trefftz basis functions by local projection constraints imposed on an underlying polynomial DG space.
The direct lineage of the present work is the embedded Trefftz method for Stokes developed in \cite{LLS_NM_2024} and its extension from Stokes to the variable-coefficient Oseen operator in Part~I \cite{SVLL1_ARXIV_2026}.
There, the higher-order pressure components are locally tied to the velocity, leading to a reduced velocity formulation for the Oseen problem, and stability is established through a local-global splitting argument combined with a convection-resolution smallness condition.
These linear results provide the foundation for the nonlinear analysis carried out here.

A first attempt to extend Trefftz-type ideas to nonlinear elliptic problems was made in the numerical study \cite{yang2020trefftz}.
A related Taylor-polynomial construction for linear elliptic problems with smooth coefficients, including a convergence analysis, is given in \cite{IMPS_IMAJNA_2025}.
The approach is based on explicit Taylor expansions of the nonlinear differential operator in order to construct local approximation functions.
In particular, it relies on pointwise Taylor data and repeated local series evaluations.
The resulting bases are typically poorly conditioned, and the available analysis is restricted to smooth settings in which the required pointwise information is meaningful.
This is conceptually distinct from the present embedded approach, where the Trefftz constraint is imposed by local projections inside an ambient DG space and conditioning is inherited from the underlying polynomial space due to orthogonality properties of the projection.

For DG discretizations of the incompressible Navier--Stokes equations, see the framework and analysis in \cite{DE10,DiPietroErn}.
In \cite{CKS05} an analysis of a LDG method is given, which is based on a reduced Oseen formulation obtained by eliminating the pressure locally, an approach we also use here.
Divergence-free DG methods provide another important comparison point, see for example \cite{montlaur2010discontinuous,montlaur08}, however, their reduction mechanism is of a different nature and is substantially weaker than the embedded Trefftz reduction considered here.
In particular, enforcing elementwise divergence-free constraints removes only the pressure-divergence part of the local algebraic structure, whereas the embedded Trefftz constraint incorporates the projected Oseen operator and ties the higher-order pressure components locally to the velocity.
This leads to a significantly stronger local elimination, but also to a much more delicate stability theory.

\paragraph{Contributions.}
The present paper provides, to the best of our knowledge, the first embedded Trefftz-DG method for a nonlinear problem and full analysis thereof.
The main difficulty is not only the nonlinear convection term itself, but the fact that the reduced Trefftz space depends on the convection field.
Hence different Picard iterates belong to different discrete spaces.
As a consequence, the nonlinear analysis cannot be obtained by a direct adaptation of the standard DG Navier--Stokes theory, nor by applying the linear Oseen theory from Part~I separately at each step.
A new mechanism is needed to compare reduced Trefftz spaces generated by different convection fields.

The main contribution of this paper is a nonlinear stability and convergence theory for the embedded Trefftz-DG discretization of the steady incompressible Navier--Stokes equations, built on the reduced Oseen theory developed in Part~I.
The key new ingredients are as follows:
\begin{itemize}
    \item We formulate the nonlinear embedded Trefftz-DG method through a Picard map based on the Oseen solver from Part~I.
    This rewrites the discrete Navier--Stokes problem as a fixed-point problem in convection-dependent reduced Trefftz spaces.

    \item We prove existence of discrete Navier--Stokes solutions by Brouwer's fixed-point theorem under a resolution condition ensuring that all Oseen iterates remain in a uniform regime of Trefftz stability.

    \item We prove uniqueness of the discrete solution and convergence of the Picard iteration by Banach's fixed-point theorem under an additional smallness assumption on the data.
    In this regime, the reduced Oseen solution map is contractive.

    \item We derive an a priori convergence analysis by relating the embedded Trefftz discretization to the underlying DG method.
    In particular, the error estimates inherit the approximation properties of any DG Navier--Stokes analysis whose discrete trilinear form satisfies our abstract assumptions.
    This includes, in particular, the LDG formulation and the DG formulation of \cite{DE10,DiPietroErn,CKS05}.
\end{itemize}

\paragraph{Outline.}
The discrete DG and embedded Trefftz-DG formulations for the Navier--Stokes problem are presented in \cref{sec:discretization}.
For the reader's convenience, the reduced Oseen formulation and the key linear results from Part~I that enter the nonlinear analysis are summarized in \cref{sec:oseen_recap}.
The main nonlinear analysis is carried out in \cref{sec:nsanalysis}, where we prove existence by Brouwer's theorem, uniqueness by Banach's theorem, and convergence of the Picard iteration.
The a priori error analysis is given in \cref{sec:ns_convergence}; it shows how the embedded Trefftz method inherits convergence estimates from the corresponding underlying DG discretization, provided the chosen trilinear form satisfies the abstract continuity and stability assumptions used in the analysis.
Numerical results, including Kovasznay flow, Reynolds-number studies, and the Sch\"afer--Turek benchmark, are reported in \cref{sec:numerics}.

\subsection{The steady Navier--Stokes problem}\label{sec:model}
Consider an open bounded
Lipschitz domain $\Omega \subset \mathbb{R}^d$ with $d=2, 3$.
The steady Navier--Stokes flow problem reads as follows: Find velocity $u: \Omega \to \mathbb{R}^d$ and pressure $p: \Omega \to \mathbb{R}$ such that
\begin{eqs}\label{eq:sns}
-\nu\Delta u + (u\cdot\nabla)u + \nabla p &= f, \quad \text{in } \Omega, \\
\Div u &= 0, \quad \text{in } \Omega, \\
u &= 0, \quad \text{on } \partial\Omega,
\end{eqs}
where $f: \Omega \to \mathbb{R}^d$ is the external body force and $\nu > 0$ is the dynamic
viscosity. 
For the ease of presentation we only consider homogeneous Dirichlet boundary conditions. 
The weak formulation of the problem \eqref{eq:sns} is
then given by: Find $(u,p)\in [H^1_0(\Omega)]^d \times L^2_0(\Omega)$
such that
\begin{equation}\label{eq:weaksns}
\begin{aligned}
\int_\Omega \nu\!~\nabla u\!:\!\nabla v \dx +\int_\Omega (u\cdot\nabla)u\cdot v \dx
- \int_\Omega \Div v\!~p \dx &= \int_\Omega f\!\cdot\! v \dx && \forall v\in [H^1_0(\Omega)]^d,  \\
-\int_\Omega  \Div u\!~q \dx&= 0  && \forall q\in L^2_0(\Omega),
\end{aligned}
\end{equation}
For simplicity we only consider homogeneous Dirichlet boundary conditions.

The central building block of our approach is the \emph{Oseen problem}, obtained by
linearizing~\eqref{eq:sns} around a given convection field $w \in [W^{1,\infty}(\Omega)]^d$:
Find velocity $u: \Omega \to \mathbb{R}^d$ and pressure $p: \Omega \to \mathbb{R}$ such that
\begin{eqs}\label{eq:oseen}
-\nu\Delta u + (w\cdot\nabla)u + \nabla p &= f, \quad \text{in } \Omega, \\
\Div u &= 0, \quad \text{in } \Omega, \\
u &= 0, \quad \text{on } \partial\Omega.
\end{eqs}
The Navier--Stokes problem~\eqref{eq:sns} is equivalent to the fixed-point problem
$u = S(u)$, where $S(w)$ denotes the solution operator of~\eqref{eq:oseen}.
The Picard iteration $u^{(n+1)} = S(u^{(n)})$ therefore reduces each step to solving an
Oseen problem with convection field $w = u^{(n)}$.
For $d=2$ and under a smallness assumption on the data, the Oseen problems are well-posed and the Picard iteration converges to a solution of the Navier--Stokes problem.

\section{DG and Trefftz-DG discretization}\label{sec:discretization}
For the steady Navier--Stokes discretization, the basic linear building block is the Oseen problem from Part~I.
The nonlinear Navier--Stokes method is obtained by replacing the frozen convection field \(w_h\) by the discrete velocity itself.
We first introduce the Trefftz-DG method for the Oseen problem by recalling the necessary spaces and operators from Part~I.

\subsection{Meshes and DG notation}\label{sec:notation_dg}

We consider a shape-regular simplicial mesh \(\Th\) of \(\Omega\), and denote by
\(\Fh=\Fhi\cup\Fhb\) the set of all interior and boundary facets.
On \(F\in\Fhi\) we use the standard average and jump operators
\(\avg{\cdot}\) and \(\jmp{\cdot}\), while on \(F\in\Fhb\) we set
\(\avg{\phi}=\phi\) and \(\jmp{\phi}=\phi\).
With the same convention as in Part~I, we write $h$ both for the global mesh size $h:=\max_{T\in\Th}h_T$ and for the local piecewise constant mesh-size functions on $\Th$, $\Fh$, and $\partial\Th$.
Thus $h|_T=h_T$ on elements, $h|_F=h_F:=\diam(F)$ on facets, and on element-boundary contributions $F\subset\partial T$ we may take $h|_F=h_T$.
Shape regularity gives $h_F\simeq h_T$ for $F\subset\partial T$.
We write $(\cdot,\cdot)_\Th:=\sum_{T\in\Th}(\cdot,\cdot)_T$.

As in Part~I, the DG space is
\begin{equation}\tag{$\IX$}\label{eq:IX}
\IX=[\IP^k]^d(\Th)\times \IP^{k-1}(\Th)/\IR.
\end{equation}

\subsection{Local operator and Trefftz space}
We set $ \IXreg:=[H^2(\Th)]^d\times H^1(\Th)$ and recall the local Oseen operator as
\begin{equation}\label{eq:BKw}\tag{$\oposeenelem{w}$}
\langle \oposeenelem{w}\up,\rs_h\rangle_T
:=
h\nu^{-\frac12}\bigl(-\nu\Delta u+w\cdot\nabla u+\nabla p,r_h\bigr)_T
+
\nu^{\frac12}\bigl(-\Div u,s_h\bigr)_T,
\end{equation}
with the local test space
\begin{equation}\tag{$\IQ$}
\IQ(T):=[\IP^{k-2}(T)]^d\times \IP^{k-1}(T),
\qquad
\IQ:=\prod_{T\in\Th}\IQ(T).
\end{equation}
for \(\up=(u,p)\in \IXreg(T)\) and \(\rs_h=(r_h,s_h)\in \IQ(T)\). 
Its global counterpart \(\oposeen{w}:\IXreg\to \IQ'\) is defined by summation over the mesh.
The corresponding Trefftz space is
\begin{equation}\label{eq:Trefftzconditions}\tag{$\IT_w$}
\IT_w
:=
\bigl\{
\up_h\in\IX:\ \langle \oposeen{w}\up_h,\rs_h\rangle_T=0
\ \ \forall\,\rs_h\in\IQ(T),\ \forall\,T\in\Th
\bigr\}.
\end{equation}

\subsection{Embedded Trefftz-DG discretization of the Oseen problem}\label{sec:oseen}
As in Part~I, we write \(\up_h=(u_h,p_h)\) and \(\vq_h=(v_h,q_h)\), and define the
standard SIPDG diffusion and pressure--velocity coupling forms by
\begin{subequations}
\begin{align}
\label{eq:sipdg}
\blfvis(u_h,v_h)
&:=
(\nu \nabla u_h,\nabla v_h)_\Th
-(\avg{\nu\partial_n u_h},\jmp{v_h})_\Fh
-(\avg{\nu\partial_n v_h},\jmp{u_h})_\Fh
+\frac{\lambda\nu}{h}(\jmp{u_h},\jmp{v_h})_\Fh,
\\
\label{eq:divdg}
\blfdiv(v_h,p_h)
&:=
-(\Div v_h,p_h)_\Th
+(\jmp{v_h\cdot n},\avg{p_h})_\Fh.
\end{align}
\end{subequations}
Here \(\lambda>0\) denotes the interior penalty parameter, chosen sufficiently large, and \(\blfconv{w}\) denotes a trilinear form discretizing the convective term \((w\cdot\nabla)u\), for which we make abstract assumptions in \cref{ass:ch}.
Suitable choices are discussed in \cite[Appendix~A]{SVLL1_ARXIV_2026}.
The Oseen bilinear form is then given by
\begin{equation}\tag{$\blfoseen{w}$}
\blfoseen{w}(\up_h,\vq_h)
:=
\blfvis(u_h,v_h)
+\blfdiv(u_h,q_h)
+\blfdiv(v_h,p_h)
+\blfconv{w}(u_h,v_h).
\end{equation}
Combining the global Trefftz equation with the local residual equation 
and the product test space 
\[
\IZ_w:=\IT_w\times \IQ, 
\]
leads to the
coupled Trefftz bilinear form
\begin{equation}\tag{$\blftrefftz{w}$}
\blftrefftz{w}(\up_h,\z_h)
:=
\blfoseen{w}(\up_h,\vq_h)
+\langle \oposeen{w}\up_h,\rs_h\rangle_\Th,
\qquad
\z_h=(\vq_h,\rs_h)\in\IZ_w,
\end{equation}
and the corresponding right-hand side
\begin{equation}\tag{$F$}
F(\z_h)
:=
(f,v_h)_\Th+(h\nu^{-\frac12}f,r_h)_\Th.
\end{equation}
With this notation, the Oseen Trefftz-DG problem can be written
compactly as: find \(\up_h\in\IX\) such that
\begin{equation}\label{eq:tdgoseen}
\blftrefftz{w_h}(\up_h,\z_h)=F(\z_h)
\qquad\forall\,\z_h\in\IZ_{w_h}.
\end{equation}

\subsection{Embedded Trefftz-DG discretization of the Navier--Stokes problem}\label{sec:ns_dg}

Using our notion from the linear Oseen problem, we can now introduce the embedded Trefftz-DG method for the steady Navier--Stokes problem.
The Trefftz-DG method to \eqref{eq:weaksns} is simply given by replacing the convective velocity $w$ in the Oseen problem by the velocity solution $u$ itself, which leads to the following nonlinear problem: 
Find $\up_h = (u_h, p_h) \in \IX$ such that
\begin{eqs}\label{eq:tdgns}
    \blftrefftz{u_h}(\up_h,\z_h) = F(\z_h) && \forall \z_h \in \IZ_{u_h}.
\end{eqs}

To solve the steady Navier--Stokes problem \eqref{eq:sns}, we employ Picard's method, a standard fixed-point iteration scheme.
Here the first part of the nonlinear term $(u\cdot\nabla)u$ is frozen from one iteration to the next, which leads to solving a sequence of Oseen problems.
For that we employ the Trefftz-DG method for the Oseen problem stated in \cref{sec:oseen}.
The scheme reads as follows:
\begin{algorithm}[H]
    \caption{Picard iteration for Trefftz-DG method for steady Navier--Stokes.}\label{alg:picard}
    \begin{algorithmic}
        \State Initialize $u^0 = 0$.
        \For{$\ell=0,1,2,\ldots$}
            \State Solve the Oseen problem \eqref{eq:tdgoseen} with $w = u^\ell$ to obtain $(u^{\ell+1},p^{\ell+1})$.
            \If{$\|u^{\ell+1} - u^\ell\| < \texttt{TOL}$}
                \State Stop and return $(u^{\ell+1},p^{\ell+1})$.
            \EndIf
        \EndFor
    \end{algorithmic}
\end{algorithm}
In the algorithm \cref{alg:picard}, we initialize the iteration with $u^0=0$, which is a common choice for the initial guess in Picard iterations for Navier--Stokes.
The \texttt{TOL} is a user-defined tolerance for convergence, and the norm $\|\cdot\|$ can be chosen appropriately, for example as the $L^2$-norm.
The analysis of the convergence of the above iteration scheme relies on the analysis of the Oseen problem, which is the subject of the next section.

\section{Recap: Oseen Trefftz-DG results from Part I}\label{sec:oseen_recap}

In this section we recall, in compact form, the Oseen Trefftz-DG results from Part I that are used below.
Since the Navier-Stokes analysis is carried out entirely in a fixed convection regime, we state these results directly in that uniform form.

\subsection{Reduced formulation}\label{sec:reduced_formulation}

We next recall the spaces and operators from the Oseen analysis in Part~I in the form needed below; see in particular \cite[Sections~3--4]{SVLL1_ARXIV_2026}.
We first introduce the full local operator and Trefftz space, and then the reduced velocity formulation obtained by eliminating the higher-order pressure component.

\paragraph{Projected operators and pressure lifting.}
To eliminate the higher-order pressure component, we recall the pressure lifting which reconstructs a higher-order pressure from the velocity. 
We first define the $L^2$ projection into the space $\IP^{k-1}(T)$ via gradients
\begin{equation} \tag{$\Pi_{\nabla,T}$} \label{eq:Pinabla}
\Pi_{\nabla,T}: [L^2(T)]^d \to \IP^{k-1}(T)/\mathbb{R}, \qquad v \mapsto \argmin_{q_h \in \IP^{k-1}(T)/\mathbb{R}} \Vert \nabla q_h - v \Vert_{T},
\end{equation}
with $\Pi_{\nabla}: L^2(\Th) \to \IP^{k-1}/\IP^0$ defined accordingly s.t. $(\Pi_{\nabla} \cdot)|_T = \Pi_{\nabla,T} \cdot$ for all $T\in\Th$.
We can now define the pressure lifting $ \Plift{w}:[H^2(\Th)]^d\to \IP^{k-1}/\IP^0,$ as
\begin{equation}\label{eq:pressure_lifting}\tag{$\Plift{w}$}
\Plift{w}v = \Pi_\nabla(\nu\Delta v - \Pi^{k-2}(w\cdot\nabla v)).
\end{equation}
Given a discrete velocity $v_h\in [\IP^k]^d$, the pressure lifting allows to construct $(v_h,\Plift{w}v_h)\in\IT_w$ whenever a suitable pair exists, i.e. to lift a suitable velocity into the Trefftz space.
More precisely, the map $\Plift{w}$ is defined for all sufficiently regular velocities.
However, the pair $(v_h,\Plift{w}v_h)$ belongs to $\IT_w$ only if the projected residual is a discrete gradient and the local divergence constraint is satisfied.

\paragraph{Reduced velocity formulation.}
We set
\[
\XRreg:=[H^2(\Th)]^d,
\]
We define the reduced local test space as the kernel of the projection \(\Pi_{\nabla,T}\):
\begin{equation}\label{eq:QR}\tag{$\QR$}
\QR(T):=\ker \Pi_{\nabla,T},
\qquad
\QR:=\prod_{T\in\Th}\QR(T).
\end{equation}
So \(\QR(T)\) is the \(L^2(T)\)-orthogonal complement of the discrete gradient space inside \([\IP^{k-2}(T)]^d\), giving the decomposition
$
[\IP^{k-2}(T)]^d
=
\nabla \IP^{k-1}(T)
\oplus^\perp
\QR(T).
$

The reduced local operator is
\begin{equation}\label{eq:ArKw}\tag{$\oposeenrelem{w}$}
\langle \oposeenrelem{w}u,r_h\rangle_T := h\nu^{-\frac12}\bigl(-\nu\Delta u+w\cdot\nabla u,r_h\bigr)_T
\qquad \forall\,u\in [H^2(T)]^d,\ r_h\in\QR(T),
\end{equation}
with global version \(\oposeenr{w}:[H^2(\Th)]^d\to \QR'\).

The reduced divergence-free velocity space is
\begin{equation}\label{eq:XR}\tag{$\XR$}
\XR := \bigl\{ u_h\in[\IP^k]^d:\ \Div u_h|_\Th=0,\ \blfdiv(u_h,q_h^0)=0 \ \forall\,q_h^0\in\IP^0 \bigr\},
\end{equation}
and the reduced Trefftz space is
\begin{align}\label{eq:TR}\tag{$\TR_w$}
\TR_w :=& \bigl\{ u_h\in \XR: \langle \oposeenrelem{w}u_h,r_h\rangle_T=0\ \forall\,r_h\in\QR(T),\ \forall\,T\in\Th \bigr\}
\\ =&\bigl\{ u_h\in \IT_w^u:\ \blfdiv(u_h,q_h^0)=0\ \forall\,q_h^0\in\IP^0 \bigr\}.
\notag
\end{align}
With \(\ZR:=\TR_w\times \QR\), the \textit{reduced Oseen problem} reads: 
Find \(\ur_h\in \XR\) such that
\begin{equation}\label{eq:redtdgoseen}
\blftrefftzr{w}(\ur_h,\zr_h)=\rmod{\ell}_h(\zr_h)
\qquad \forall\,\zr_h=(\vr_\IT,\rr_h)\in \ZR,
\end{equation}
where
\begin{align} \tag{$\blftrefftzr{w}$}
&\blftrefftzr{w}(\ur_h,\zr_h)
:=
\blfoseenr{w}(\ur_h,\vr_\IT)
+
\langle \oposeenr{w}\ur_h,\rr_h\rangle_\Th\\
\tag{$\blfoseenr{w}$}
&\blfoseenr{w}(u,v)
:=
\blfoseen{w}\bigl((u,\Plift{w}u),(v,\Plift{w}v)\bigr),
\end{align}
and
\begin{equation}\label{eq:rhs_functional}\tag{$\rmod{\ell}_h$}
\rmod{\ell}_h(\vr_\IT,\rr_h) := (f,\vr_\IT)_\Th - \blfdiv(\vr_\IT,\Pi_\nabla f) + (f,h\nu^{-\frac12}\rr_h)_\Th .
\end{equation}
Finally, we recall the equivalence result from Part~I:

\begin{theorem}[Equivalence of full and reduced Oseen formulations]\label{cor:uniquesol}
For every \(w_h\in\IW_{\bar\gamma}\), the full coupled Trefftz formulation and the
reduced problem \eqref{eq:redtdgoseen} are equivalent:
any solution \((u_h,p_h)\in\IX\) of the full Oseen Trefftz problem has
\(u_h\in\XR\) solving \eqref{eq:redtdgoseen}, and conversely every solution
\(\ur_h\in\XR\) of \eqref{eq:redtdgoseen} lifts to a solution
\((u_h,p_h)\in\IX\) of the full coupled formulation.
\end{theorem}

\subsection{Norms}

\paragraph{Norms on the full space.} 
We denote by \(\|\cdot\|_\Th\) the broken \(L^2\) norm, and by \(\|\cdot\|_{\Fh}\) the \(L^2\) norm on the skeleton.
For \(u\in [H^1(\Th)]^d\), \(v\in [H^2(\Th)]^d\), and \(p\in H^1(\Th)\), we use the mesh-dependent norms
\begin{align*}
\|u\|_{1,h}^2
&:=
\|\nabla u\|_\Th^2+\|h^{-\frac12}\jmp{u}\|_{\Fh}^2,
\qquad
\|v\|_{1,h,\ast}^2
:=
\|v\|_{1,h}^2
+\|h^{\frac12}\partial_n v\|_{\partial\Th}^2
+\|h \Delta v\|_\Th^2,
\\
\|p\|_{0,h}^2
&:=
\|h \nabla p\|_\Th^2
+\|h^{\frac12}\jmp{\Pi^0 p}\|_{\Fhi}^2,
\end{align*}
and, for \(\up=(u,p)\),
\begin{equation*}
\dgnorm{\up}^2:=\nu\|u\|_{1,h}^2+\nu^{-1}\|p\|_{0,h}^2,
\qquad
\dgsnorm{\up}^2:=\nu\|u\|_{1,h,\ast}^2+\nu^{-1}\|p\|_{0,h}^2.
\end{equation*}
Throughout, \(\Pi^\ell_S\) denotes the \(L^2\)-orthogonal projection onto
\(\IP^\ell(S)\) for scalar functions, or onto \([\IP^\ell(S)]^d\) for vector-valued functions; in particular, \(\Pi^0\) denotes the elementwise mean.
For \(\rs_h=(r_h,s_h)\in\IQ\), we set
\begin{equation*}%
\|\rs_h\|_{\IQ}^2:=\|r_h\|_\Th^2+\|s_h\|_\Th^2,
\end{equation*}
and let \(\|\cdot\|_{\IQ'}\) denote the induced dual norm.

\paragraph{Norms on the reduced space.} 
We set
\[
\|u\|_{\XR}:=\nu^{\frac12}\|u\|_{1,h},
\qquad
\|u\|_{\XR,\ast}:=\nu^{\frac12}\|u\|_{1,h,\ast},
\]
We equip the reduced test spaces $\QR$ and \(\ZR=\TR_w\times\QR\) with the norms
\begin{equation*}
    \|\rr_h\|_{\QR} := \|r_h\|_\Th,
    \qquad 
    \|(\vr_\IT,\rr_h)\|_{\ZR}^2 := \|\vr_\IT\|_{\XR}^2+\|\rr_h\|_{\QR}^2.
\end{equation*}

\subsection{Key analytical results from the Oseen analysis}
We define the following quantity to measure the size of the convection field $w$ in the discrete setting
\begin{equation*}%
|w|_{h,d}
:=
\nu^{-1}\Big(
\max_{T\in \Th} C_4 \|h^{1-\frac{d}{4}} w\|_{L^4(T)}
+
\max_{T\in \Th} C_{W,4}\|h^{2-\frac{d}{4}} \nabla w\|_{L^4(T)}
+
\max_{F\in \Fh} C_{J,4}\|h^{\frac14}\jmp{w\cdot n}\|_{L^4(F)}
\Big),
\end{equation*}
with constants from local discrete embeddings defined in \cite{SVLL1_ARXIV_2026}.
This norm is the same as the one used in Part~I, but whenever the convection field is a polynomial, the second term is controlled by the first.
The quantity $|w|_{h,d}$ is weaker than the discrete $H^1$ norm by a factor $h^{1-\frac{d}{4}}$:
\begin{equation}\label{eq:gamma_norm_bound}
    |w|_{h,d} \leq \nu^{-1} C_\gamma h^{1-\frac{d}{4}} \| w\|_{1,h} \qquad \forall w \in [\IP^k]^d,
\end{equation}

We restrict convection fields to the permissable set of $\IW_{\bar\gamma}$ given by
\begin{equation}\label{eq:Wgammabar}
\IW:= [\IP^k(\Th)]^d\qquad \text{and} \qquad \IW_{\bar\gamma} := \{w_h\in\IW:\ |w_h|_{h,d}\le \bar\gamma\},
\end{equation}
where we assume the resolution condition
\begin{equation}\label{ass:wh}
    \bar\gamma<\stablocal{0},
\end{equation}
which is necessary for stability of the local problems, see \cref{cor:oseen_local} below.
Here, $\stablocal{0}$ is the stability constant in $ \|\oposeenelem{0}\up_h\|_{T} \ge \stablocal{0}\,\dgnorm{\up_h}.$
Contrary to Part~I, we already restrict the permissible convection to polynomials, this is sufficient for convection appearing in Picard iteration.

We shall repeatedly use the following bound: for every
\(w\in [L^4(\Th)]^d\) and \(v\in [H^1(\Th)]^d\),
\begin{equation}\label{eq:local_conv_proj}
\|h\nu^{-\frac12}\Pi^{k-2}(w\cdot\nabla v)\|_{\Th}
\le
\nu^{\frac12}|w|_{h,d}\|\nabla v\|_{\Th}
\le
\nu^{\frac12}|w|_{h,d}\|v\|_{1,h}.
\end{equation}
by Hölder's inequality and the local \(L^4\)-inverse estimate on
\([\IP^{k-2}(T)]^d\), summed over the mesh.
Since \(\QR(T)\subset [\IP^{k-2}(T)]^d\), this also implies the bound
\begin{equation}\label{eq:local_conv_red_diff}
\|(\oposeenr{w_1}-\oposeenr{w_2})v\|_{\QR'}
\le
|w_1-w_2|_{h,d}\|v\|_{\XR}
\qquad \forall\, v\in\XR,
\end{equation}

The analysis of the method only relies on the following abstract assumptions on the trilinear form which are fulfilled for the standard choices.
\begin{assumption}[\emph{(Conditions on $\blfconv{w}$)}]\label{ass:ch}
There exists $\contconv>0$ (depending only on shape-regularity and $k$) such that for all $w\in \IW$, $u\in [H^2(\Th)]^d$, and $v_h\in[\IP^k]^d$,
\begin{equation}\label{eq:ch_cont}
  |\blfconv{w}(u,v_h)|
  \le
  \contconv\,\|w\|_{1,h}\,\|u\|_{1,h,\ast}\,\|v_h\|_{1,h}.
\end{equation}
We additionally assume
\begin{equation}\label{eq:ch_nonneg}
\blfconv{w}(u_h,u_h) \ge 0\qquad\forall w\in\IW,\ \forall u_h\in[\IP^k]^d,
\end{equation}
and the mixed $L^4$-estimate
\begin{equation}\label{eq:ch_mixed_L4}
|\blfconv{w}(u_\IL,v_{\IT})|
\le C_{t,\ast}\nu|w|_{h,d}\|u_\IL\|_{1,h,\ast}\|v_{\IT}\|_{1,h},
\qquad \forall w\in\IW,\ \forall v_{\IT}\in\IT_w^u,\ \forall u_\IL\in \IL.
\end{equation}
\end{assumption}
To discretize the convective term $(w\cdot\nabla)u$ many choices are possible. 
For now, we keep the discretization of the convection term abstract, and denote it as trilinear form $\blfconv{w}$. 
Concrete choices for the trilinear forms can be found in \cite[Appendix~A]{SVLL1_ARXIV_2026}.

We now recall continuity, local- and global-stablilty from Part~I, cf.\ \cite[Sections~3--5]{SVLL1_ARXIV_2026}.
Constants are denoted with a \(\ast\) to indicate that they are uniform on \(\IW_{\bar\gamma}\), and depend only on \(k\), the DG parameters, the shape-regularity, \(\Omega\), and \(\bar\gamma\).
Throughout Part~II, we keep explicit only the stability constants $\stabtrefftzr{*}$ and $\stabtrefftz{w_h}$, as well as the constants $\stabsolr{*}$, $\banach$, and the pointwise resolution constant $\Cres{w_h}$ we introduce later on.
The dependence on the convection field and the viscosity is made explicit only through $\|w_h\|_{1,h}/\nu$.
All remaining structural constants are absorbed into a generic constant $C>0$, except in the paragraph \cref{rem:oseen_constants}.

\begin{corollary}[Uniform local stability] \label{cor:oseen_local}
There exist local spaces $\IL\subset\IX$ and $\LR\subset\XR$ that allow to decompose 
\[
    \IX=\IT_{w_h}\oplus \IL
    \quad \text{and}\quad
    \XR=\TR_{w_h}\oplus \LR,
    \qquad \forall w_h\in\IW_{\bar\gamma}.
\]
Furthermore, for every \(w_h\in\IW_{\bar\gamma}\), the local operator \(\oposeen{w_h}:\IXreg\to\IQ'\) is continuous and 
its restriction \(\oposeen{w_h}:\IL\to\IQ'\) is invertible with
\begin{equation}\label{eq:oseen_local}
\stablocal{\ast}\,\dgnorm{\up_h}
\le
\|\oposeen{w_h}\up_h\|_{\IQ'}
\qquad \forall\,\up_h\in\IL,
\qquad
\|\oposeen{w_h}\up\|_{\IQ'}
\le
\contlocal{\ast}\,\dgsnorm{\up}
\qquad \forall\,\up\in\IXreg,
\end{equation}
and similarly for the reduced local operator \(\oposeenr{w_h}:\LR\to\QR'\) we have that
\begin{equation}\label{eq:oseen_local_red}
\stablocal{\ast}\,\|\ur_\IL\|_{\XR}
\le
\|\oposeenr{w_h}\ur_\IL\|_{\QR'}
\qquad \forall\,\ur_\IL\in\LR,
\qquad
\|\oposeenr{w_h}\ur\|_{\QR'}
\le
\contlocal{\ast}\,\|\ur\|_{\XR,\ast}
\qquad \forall\,\ur\in\XRreg.
\end{equation}
\end{corollary}

\begin{lemma}[Continuity]\label{lem:oseen_cont}
For every \(w_h\in\IW_{\bar\gamma}\), there hold
\begin{subequations}
\begin{align}
|\blfoseen{w_h}(\up,\vq_h)|
&\le
\contoseen{w_h}\,\dgsnorm{\up}\,\dgnorm{\vq_h}
&& \forall\,\up\in\IXreg,\ \vq_h\in\IX, \label{eq:oseen_cont_full}
\\
|\blfoseenr{w_h}(\ur,\vr_h)|
&\le
\contoseenr{w_h}\,\|\ur\|_{\XR,\ast}\,\|\vr_h\|_{\XR}
&& \forall\,\ur\in\XRreg,\ \vr_h\in\XR, \label{eq:oseen_cont_red}
\\
|\blfoseenr{w_h}(\ur_\IL,\vr_\IT)|
&\le
\contcross{\ast}\,\|\ur_\IL\|_{\XR}\,\|\vr_\IT\|_{\XR}
&& \forall\,\ur_\IL\in\LR,\ \vr_\IT\in\TR_{w_h}. \label{eq:oseen_cross}
\end{align}
\end{subequations}
Here \(\contoseen{w_h}\) and \(\contoseenr{w_h}\) are the continuity constants from
Part~I, with explicit dependence on \(w_h\) as recalled below; in particular they
may depend on \(\|w_h\|_{1,h}/\nu\), whereas \(\contcross{\ast}\) is uniform on
\(\IW_{\bar\gamma}\), since the cross bound depends on \(w_h\) only through
\(|w_h|_{h,d}\).
\end{lemma}

\begin{theorem}[Uniform Oseen stability in the \texorpdfstring{$\bar\gamma$}{gammabar}-regime]
\label{thm:oseen_stability}
Assume that the penalty parameter $\lambda$ is chosen sufficiently large,
depending only on $k$, the DG parameters, the shape-regularity, and
$\bar\gamma$.
Then, for every $w_h\in\IW_{\bar\gamma}$, the reduced and full Oseen
Trefftz formulations from Part~I are well posed. More precisely,
there exist constants $\stabtrefftzr{*}>0$ and $\stabsolr{*}>0$,
depending only on $k$, the DG parameters, the shape-regularity, $\Omega$,
and $\bar\gamma$, such that
\begin{subequations}
\begin{align}
\blfoseenr{w_h}(\ur_\IT,\ur_\IT)
&\ge
C\,\|\ur_\IT\|_{\XR}^2
&& \forall\,\ur_\IT\in\TR_{w_h},
\label{eq:oseen_coerc}
\\
\sup_{\zr_h\in\ZR\setminus\{0\}}
\frac{\blftrefftzr{w_h}(\ur_h,\zr_h)}{\|\zr_h\|_{\ZR}}
&\ge
\stabtrefftzr{*}\,\|\ur_h\|_{\XR}
&& \forall\,\ur_h\in\XR,
\label{eq:oseen_red_infsup}
\\
\sup_{\z_h\in \IZ_{w_h}\setminus\{0\}}
\frac{\blftrefftz{w_h}(\up_h,\z_h)}{\|\z_h\|_{\IZ}}
&\ge
\stabtrefftz{w_h}\,\dgnorm{\up_h}
&& \forall\,\up_h\in\IX,
\label{eq:oseen_full_infsup}
\\
\|S_r(w_h)\|_{\XR}
&\le
\stabsolr{*}\,\nu^{-\frac12}\|f\|_{\Th}.
\label{eq:oseen_energy}
\end{align}
\end{subequations}
The precise origin of these constants from the Part~I analysis is recalled
in the paragraph below.
\end{theorem}

\begin{remark}[On the origin of the uniform constants]\label{rem:oseen_constants}
For completeness, we briefly indicate how the constants used in \cref{cor:oseen_local,lem:oseen_cont,thm:oseen_stability} are inherited from Part~I.

The uniform local constants are obtained from the prototype Stokes constants:
\[
\stablocal{\ast}=\stablocal{0}-\bar\gamma,
\qquad
\contlocal{\ast}=\bigl(2\contlocal{0}^2+2\bar\gamma^2\bigr)^{1/2},
\]
and
\[
\contcross{\ast}:=\sup_{w_h\in\IW_{\bar\gamma}}\contcross{w_h}.
\]
Here $\contcross{w_h}$ depends on $w_h$ only through $|w_h|_{h,d}$, and is therefore uniform on $\IW_{\bar\gamma}$.
Moreover, once $\lambda$ is chosen sufficiently large, the coercivity constant on the reduced Trefftz part can be fixed uniformly, and one may take $\staboseenr{\ast}=\tfrac12$.

Accordingly, in the reduced inf-sup constant $\stabtrefftzr{*}$ obtained in Part~I  the convection only enters through $\staboseenr{\ast}$, $\contcross{\ast}$, and $\stablocal{\ast}$, and is therefore uniform on $\IW_{\bar\gamma}$.
Likewise, the full inf-sup constant $\stabtrefftz{w_h}$ from the full coupled inf-sup estimate in Part~I is based on the piecewise-constant LBB constant $\alpha_0$ and the pointwise Oseen continuity constant $\contoseen{w_h}$.

Finally, the only non-uniform dependence on the convection field enters through the pointwise continuity constants:
\[
\contoseen{w_h}\lesssim 1+\frac{\|w_h\|_{1,h}}{\nu},
\qquad
\contoseenr{w_h}\lesssim \contoseen{w_h},
\]
up to structural multiplicative constants depending only on $k$, the DG parameters, and the shape-regularity.
The constant $\contoseen{w_h}$ appears multiplicative in the full inf-sup constant $\stabtrefftz{w_h}$.
Hence, the full continuity and full inf-sup constants remain dependent on $w_h$ until the Banach regime yields a uniform bound on $\|w_h\|_{1,h}/\nu$.
\end{remark}

\section{Analysis of the Navier--Stokes Trefftz-DG method}\label{sec:nsanalysis}

In view of the reduced Oseen formulation from \cref{sec:reduced_formulation}, we formulate the nonlinear fixed-point argument directly on the reduced velocity space $\XR$.
For a given convection field $w_h\in\XR$, let
\[
S_r(w_h)\in \XR
\]
denote the unique solution of the reduced Oseen problem \eqref{eq:redtdgoseen}
with convection field $w_h$.

\subsection{Existence of a solution via Brouwer's theorem}

\begin{theorem}[Existence for the reduced Navier--Stokes problem via Brouwer]
\label{thm:NS_existence_Brouwer}
Fix $\bar\gamma\in(0,1)$ and assume that \cref{ass:wh} hold uniformly for all $w_h\in \IW_{\bar\gamma}$, with the penalty parameter chosen large enough.
For each $w_h\in \IW_{\bar\gamma}$, let $S_r(w_h)\in\XR$ denote the unique solution of the
reduced Oseen problem \eqref{eq:redtdgoseen}.

Assume that the mesh resolves the data in the sense that
\begin{equation}\label{eq:resolution_condition_brouwer}
    C_\gamma \stabsolr{*} h^{1-\frac d4} \frac{\|f\|_{\Th}}{\nu^2}
\ \le\ \bar\gamma,
\end{equation}
where $C_\gamma$ is the constant from \eqref{eq:gamma_norm_bound}.
Then there exists $u_h\in \IW_{\bar\gamma}$ such that
\[
u_h=S_r(u_h).
\]
In particular, $u_h\in\XR$ solves the reduced Navier--Stokes problem, i.e.
\eqref{eq:redtdgoseen} with convection field $w_h=u_h$.

Moreover, by \cref{cor:uniquesol}, there exists $p_h\in \IP^{k-1}(\Th)/\IR$ such that
$(u_h,p_h)\in \IX$ solves the full embedded Trefftz--DG Navier--Stokes formulation
\eqref{eq:tdgns}.
\end{theorem}

\begin{proof}
We work on the finite-dimensional space $\IW_{\bar\gamma}$, equipped with the norm $|\cdot|_{h,d}$.

\emph{Step 1: $\IW_{\bar\gamma}$ is nonempty, convex, and compact.}
Clearly $0\in \IW_{\bar\gamma}$, hence $\IW_{\bar\gamma}$ is nonempty.
Since $|\cdot|_{h,d}$ is a norm on the finite-dimensional space $\IW$, the set $\IW_{\bar\gamma}$ is closed, bounded, and convex.
Therefore it is compact.

\emph{Step 2: positivity of $\stabtrefftzr{*}$.}
By \cref{cor:oseen_local}, for every $w_h\in \IW_{\bar\gamma}$ there holds $\stabtrefftzr{w_h}>0$.
Moreover, the map $w_h\mapsto \stabtrefftzr{w_h}$ is continuous on $\IW_{\bar\gamma}$,
since $\stabtrefftzr{w_h}$ is given explicitly in terms of the constants
$\stablocal{w_h}$, $\contlocal{w_h}$, and $\contoseenr{w_h}$, all of which depend continuously on $w_h$.
Since $\IW_{\bar\gamma}$ is compact, the infimum is strictly positive
\[
\stabtrefftzr{*}>0.
\]

\emph{Step 3: invariance $S_r(\IW_{\bar\gamma})\subset \IW_{\bar\gamma}$.}
Let $w_h\in \IW_{\bar\gamma}$.
Using \eqref{eq:gamma_norm_bound} and \eqref{eq:oseen_energy} we obtain
\[
|S_r(w_h)|_{h,d}
\le \nu^{-1} C_\gamma h^{1-\frac d4}\|S_r(w_h)\|_{1,h}
\le
C_\gamma \stabsolr{*} h^{1-\frac d4} \frac{\|f\|_{\Th}}{\nu^2}.
\]
Hence, by the resolution condition \eqref{eq:resolution_condition_brouwer},
\[
|S_r(w_h)|_{h,d}\le \bar\gamma,
\]
that is, $S_r(w_h)\in \IW_{\bar\gamma}$.

\emph{Step 4: continuity of $S_r$ on $\IW_{\bar\gamma}$.}
Since $\oposeenr{w_h}|_{\LR}:\LR\to\QR'$ is uniformly invertible for all $w_h\in\IW_{\bar\gamma}$, the projector
\[
P_{\LR}(w_h):=\oposeenr{w_h}^{-1}\oposeenr{w_h}:\XR\to\LR,
\]
is well-defined and depends continuously on $w_h$; hence also
$P_{\TR}(w_h):=\Id-P_{\LR}(w_h):\XR\to\TR_{w_h}$ depends continuously on $w_h$.
Thus the reduced problem \eqref{eq:redtdgoseen} is a continuously parameter-dependent family of linear systems on the fixed finite-dimensional space $\XR$, 
and \cref{eq:oseen_red_infsup} implies uniform invertibility on $\IW_{\bar\gamma}$.
Therefore $S_r:\IW_{\bar\gamma}\to\XR$ is continuous.

\emph{Step 5: apply Brouwer.}
By Steps 1--4, $S_r$ is a continuous self-map on the nonempty convex compact set $\IW_{\bar\gamma}$.
The finite-dimensional Brouwer fixed-point theorem \cite[\S2.3]{Zeidler1986FixedPoint} therefore yields $u_h\in\IW_{\bar\gamma}$ such that
\[
u_h=S_r(u_h).
\]
By definition of $S_r$, this means exactly that $u_h\in\XR$ solves the reduced Navier--Stokes problem.

Finally, \cref{cor:uniquesol} implies that there exists $p_h\in \IP^{k-1}(\Th)/\IR$
such that $(u_h,p_h)\in\IX$ solves the full formulation \eqref{eq:tdgns}.
\end{proof}

\subsection{Uniqueness of the solution via Banach's theorem}
We now turn to uniqueness.
In contrast to the Brouwer argument, this requires a genuine data smallness condition.
We keep this separate from the resolution condition \eqref{eq:resolution_condition_brouwer}.

A central difficulty is that the reduced Trefftz space $\TR_w$ depends on the convecting velocity $w$. 
Since we fixed the local complement space $\LR$ independent of $w$, the
space decomposition
\[
\XR=\TR_w\oplus \LR
\]
allows us to transfer functions between reduced Trefftz spaces by projecting
along $\LR$.

\begin{lemma}[Stability of the projection between reduced Trefftz spaces]
\label{lem:PiTR_stability}
Let $w_1,w_2\in \IW_{\bar\gamma}$ and define
\[
\Pi_{\TR}^{w_1\to w_2}:\TR_{w_1}\to \TR_{w_2},
\qquad
\Pi_{\TR}^{w_1\to w_2}(\vr_h)
:=
\vr_h-\oposeenr{w_2}^{-1}\oposeenr{w_2}\vr_h.
\]
Then $\Pi_{\TR}^{w_1\to w_2}(\vr_h)\in \TR_{w_2}$ for all $\vr_h\in\TR_{w_1}$ and
\begin{equation}\label{eq:PiTR_stability}
\|\Pi_{\TR}^{w_1\to w_2}\vr_h-\vr_h\|_{\XR}
\le
\frac{1}{\stablocal{w_2}}
|w_1-w_2|_{h,d} 
\|\vr_h\|_{\XR}.
\end{equation}
\end{lemma}
\begin{proof}
First, we recall that $\oposeenr{w_2}$ maps from $\XR$ (to $\QR'$) while  
$\oposeenr{w_2}^{-1}= (\oposeenr{w_2}|_{\LR})^{-1}$ maps (from $\QR'$) to $\LR$ so that $\oposeenr{w_2}^{-1}\oposeenr{w_2}: \XR \to \LR$.
Hence, the definition of
$\Pi_{\TR}^{w_1\to w_2}$ is meaningful, and
\[ \tag{$\ast$}
\Pi_{\TR}^{w_1\to w_2}\vr_h-\vr_h
=
-\oposeenr{w_2}^{-1}\oposeenr{w_2}\vr_h
\in \LR \text{ for all } \vr_h \in \TR_{w_1} \subset \XR.
\]
Moreover, $\oposeenr{w_2}\oposeenr{w_2}^{-1} = \id_{\QR'}$ so that $\oposeenr{w_2}\oposeenr{w_2}^{-1}\oposeenr{w_2} = \oposeenr{w_2}$ and hence $\oposeenr{w_2} \Pi_{\TR}^{w_1\to w_2} \vr_h = \oposeenr{w_2} \vr_h - \oposeenr{w_2} \vr_h = 0$ for all $\vr_h \in \TR_{w_1}$. 
Hence $\Pi_{\TR}^{w_1\to w_2}\vr_h\in \XR\cap \ker(\oposeenr{w_2}) = \TR_{w_2}$.
For the stability bound, 
we use $(\ast)$ and the stability bound \eqref{eq:oseen_local_red} to get for arbitrary $\vr_h \in \TR_{w_1}$
\begin{align*}
\stablocal{w_2}
&\|\Pi_{\TR}^{w_1\to w_2}\vr_h-\vr_h\|_{\XR}
 \le
\|\oposeenr{w_2}(\Pi_{\TR}^{w_1\to w_2}\vr_h-\vr_h)\|_{\QR'} 
=
\| -\oposeenr{w_2} \vr_h \|_{\QR'}
=
\|(\oposeenr{w_1} - \oposeenr{w_2})\vr_h\|_{\QR'} \\
&\le
|w_1-w_2|_{h,d}\,\|\vr_h\|_{\XR}.
\end{align*}
where we made use of $\oposeenr{w_2}\Pi_{\TR}^{w_1\to w_2}\vr_h=0$
and $\oposeenr{w_1}\vr_h = 0$ and a local convection estimate for the convection difference as in \eqref{eq:local_conv_red_diff}.
Dividing by $\stablocal{w_2}$ concludes the proof.
\end{proof}

\begin{lemma}[Lipschitz bound for the reduced Oseen map]
\label{lem:Sr_Lip_Banach}
Assume that $\blfconv{w}(\cdot,\cdot)$ satisfies \cref{ass:ch}. 
Let $w_1,w_2\in \IW_{\bar\gamma}$ and set
\[
\ur_i:=S_r(w_i)\in\XR,
\qquad i=1,2.
\]
Then there holds 
\begin{subequations}\label{eq:Sr_Lip_all}
\begin{align}\label{eq:Sr_Lip}
\norm{\ur_1-\ur_2}_{1,h}
&\le
\frac{\norm{f}_{\Th}}{\nu^{2}}
\Bigl(
\banach\,\vdgnorm{w_1-w_2}
+
\Cres{w_1}\,\nu|w_1-w_2|_{h,d}
\Bigr) \\
  \label{eq:Lambda_B}
\text{with}\quad    \banach
&:=
C \, 
\contconv \, \stabsolr{*} \, \stabtrefftzr{*}^{-1}
\quad\text{and}\quad
\Cres{w_1}
:=
C \stabtrefftzr{*}^{-1} 
\left(
    \stablocal{\ast}^{-1} + \stabsolr{*} ( 1 + \contoseenr{w_1} \stablocal{\ast}^{-1})
\right)
\end{align}
\end{subequations}
for a generic constant $C>0$ that depends only on $k$, the DG parameters and the shape-regularity. 
The estimate is pointwise in $w_1,w_2$; in particular, the second named constant keeps the explicit dependence on \(\contoseenr{w_1}\).
\end{lemma}
\begin{proof}
In the following we will bound the left hand side term by several parts that we will bound by factors of 
$$
\circled{A} := \nu^{-2} \norm{f}_{\Th} \vdgnorm{w_1 - w_2}
\quad \text{and} \quad
\circled{B} := \nu^{-2} \norm{f}_{\Th} \nu |w_1 - w_2|_{h,d},
$$
so that \eqref{eq:Sr_Lip_all} reads $\vdgnorm{\rmod{u}_1-\rmod{u}_2}
\leq \banach \cdot \circled{A} + \Cres{w_1} \cdot \circled{B}$. 
Set
$ e_h:=\ur_1-\ur_2\in\XR. $
We apply the reduced inf-sup estimate \eqref{eq:oseen_red_infsup} with convection field
$w_2$
\[
\stabtrefftzr{*} \norm{e_h}_{\XR}
\le
\stabtrefftzr{w_2} \norm{e_h}_{\XR}
\le
\sup_{\zr_2\in\ZR\setminus\{0\}}
\frac{\blftrefftzr{w_2}(e_h,\zr_2)}{\norm{\zr_2}_{\ZR}},
\]
where $\ZR=\TR_{w_2}\times\QR$.
Fix $\zr_2=(\vr_2,\rr_h)\in\ZR$, and define the transferred Trefftz test
\[
\vr_1:=\Pi_{\TR}^{w_2\to w_1}\vr_2\in \TR_{w_1},
\qquad
\zr_1:=(\vr_1,\rr_h)\in \TR_{w_1}\times\QR.
\]
Since $\ur_i$ solves the reduced Oseen problem with convection field $w_i$, we have
$\blftrefftzr{w_2}(\ur_2,\zr_2)=\rmod{\ell}_h(\vr_2,\rr_h)$,
and
$\blftrefftzr{w_1}(\ur_1,\zr_1)=\rmod{\ell}_h(\vr_1,\rr_h)$,
and therefore
\begin{equation*}%
\blftrefftzr{w_2}(e_h,\zr_2)
=
\underbrace{\rmod{\ell}_h(\vr_1,\rr_h)-\rmod{\ell}_h(\vr_2,\rr_h)}_{=:I}
+
\underbrace{\Bigl(
\blftrefftzr{w_2}(\ur_1,\zr_2)-\blftrefftzr{w_1}(\ur_1,\zr_1)
\Bigr)}_{=:II}.
\end{equation*}
\noindent
\emph{Bound I.}
By definition of $\rmod{\ell}_h$, cf. \eqref{eq:rhs_functional}, we find 
$
I
=
(f,\vr_1-\vr_2)_\Th-\blfdiv(\vr_1-\vr_2,\Pi_\nabla f)$.
Using the discrete Poincar\'e inequality, continuity of $\blfdiv$, and
\cref{lem:PiTR_stability} and $\stablocal{w_1}^{-1}\le \stablocal{*}^{-1}$ as $w_1\in \IW_{\bar\gamma}$ we obtain
\begin{align*} %
|I|
&\!\le\!
C\nu^{-\frac12}\norm{f}_{\Th} \norm{\vr_1-\vr_2}_{\XR}
\!\le\!
C\stablocal{\ast}^{-1}\nu^{-\frac32}\norm{f}_{\Th} 
\, \nu |w_1-w_2|_{h,d} \norm{\vr_2}_{\XR} 
=
C\stablocal{\ast}^{-1}\nu^{\frac12} \norm{\vr_2}_{\XR} \cdot \circled{B}.
\end{align*}
\medskip\noindent
\emph{Bound II.}
Adding and subtracting $\blfoseenr{w_1}(\ur_1,\vr_2)$ we write
\[
II
=
\underbrace{
\blfoseenr{w_2}(\ur_1,\vr_2)-\blfoseenr{w_1}(\ur_1,\vr_2)
}_{=:II_a}
+
\underbrace{
\blfoseenr{w_1}(\ur_1,\vr_2-\vr_1)
}_{=:II_b}
+
\underbrace{
\langle (\oposeenr{w_2}-\oposeenr{w_1})\ur_1,\rr_h\rangle_\Th
}_{=:II_c}.
\]
For the local term, \eqref{eq:local_conv_red_diff} together with \eqref{eq:oseen_energy} gives
\begin{equation*}%
|II_c|
\!\le
|w_1\!-\!w_2|_{h,d} \norm{\ur_1}_{\XR} \norm{\rr_h}_{\QR}
\!\le
\stabsolr{*} \nu^{-\frac12} \norm{f}_{\Th} |w_1\!-\!w_2|_{h,d} \norm{\rr_h}_{\QR}
\!=
\stabsolr{*} \nu^{\frac12} \norm{\rr_h}_{\QR} \cdot \circled{B}.
\end{equation*}
Then, $\vr_2-\vr_1\in \LR$, continuity of $\blfoseenr{w_1}$, norm equivalence on
the discrete space, \cref{lem:PiTR_stability} and \eqref{eq:oseen_energy} imply
\begin{align*}
|II_b| & = \blfoseenr{w_1}(\ur_1,\vr_2-\vr_1)
\le C \contoseenr{w_1}\norm{\ur_1}_{\XR}\norm{\vr_2-\vr_1}_{\XR}
\le C \contoseenr{w_1} \stablocal{w_1}^{-1} \norm{\ur_1}_{\XR} |w_1-w_2|_{h,d} \norm{\vr_2}_{\XR} \nonumber \\
       &\le C \stabsolr{*} \contoseenr{w_1} \stablocal{\ast}^{-1} \nu^{-\frac12} \norm{f}_{\Th} |w_1-w_2|_{h,d} \norm{\vr_2}_{\XR}
=
C \stabsolr{*} \contoseenr{w_1} \stablocal{\ast}^{-1} \nu^{\frac12} \norm{\vr_2}_{\XR} \cdot \circled{B}.
\end{align*}
It remains to estimate $II_a$.
Using the stability of $\nabla\Pi_\nabla$ and \eqref{eq:local_conv_proj}, we obtain
\begin{equation*}
\norm{\Plift{w_2}v-\Plift{w_1}v}_{0,h}
=\norm{h_{\calT} \nabla \Pi_\nabla \Pi^{k-2}\bigl((w_2-w_1)\cdot\nabla v\bigr)}_\Th
\le
C\,\nu^{\frac12}\,|w_1-w_2|_{h,d} \norm{v}_{\XR}
\qquad \forall v\in \XR.
\end{equation*}
We now split $II_a$ into its convection difference $\blfconv{w_2-w_1}(\ur_1,\vr_2)$ 
and the two pressure-lifting mismatch terms
\[
\blfdiv(\ur_1,\Plift{w_2}\vr_2-\Plift{w_1}\vr_2),
\qquad
\blfdiv(\vr_2,\Plift{w_2}\ur_1-\Plift{w_1}\ur_1).
\]
By the previous bound, \eqref{eq:oseen_energy} and continuity of $\blfdiv(\cdot,\cdot)$, both mismatch
terms satisfy
\begin{align*}%
\bigl|\blfdiv(\ur_1,\Plift{w_2}\vr_2-\Plift{w_1}\vr_2)\bigr|
& +
\bigl|\blfdiv(\vr_2,\Plift{w_2}\ur_1-\Plift{w_1}\ur_1)\bigr|
\le
C \,|w_1-w_2|_{h,d} \norm{\ur_1}_{\XR} \norm{\vr_2}_{\XR} \nonumber \\
& \le 
C \stabsolr{*} \nu^{-\frac12} \norm{f}_{\Th} |w_1-w_2|_{h,d} \norm{\vr_2}_{\XR}
= 
C \stabsolr{*} \nu^{\frac12} \norm{\vr_2}_{\XR} \cdot \circled{B}.
\end{align*}
For the convection difference
use \eqref{eq:ch_cont}: 
\begin{align*}
\bigl|\blfconv{w_2-w_1}(\ur_{1},\vr_2)\bigr|
&\le
\contconv \norm{w_2-w_1}_{1,h} \norm{\ur_{1}}_{1,h} \norm{\vr_2}_{1,h}
\notag
\\
&\le
\contconv
\stabsolr{*} \nu^{-\frac32} 
\norm{f}_{\Th} 
\norm{w_2-w_1}_{1,h} \norm{\vr_2}_{\XR}
=
\contconv
\stabsolr{*} \nu^{\frac12} 
\norm{\vr_2}_{\XR} \cdot \circled{A}.
\label{eq:Banach_term_bound_new}
\end{align*}
Hence,
$$
II_a \leq 
\contconv
\stabsolr{*} \nu^{\frac12} 
\norm{\vr_2}_{\XR} \cdot \circled{A}
+ 
C \stabsolr{*} \nu^{\frac12} \norm{\vr_2}_{\XR} \cdot \circled{B}.
$$
Collecting all terms, 
and using the inf-sup estimate 
and 
$\norm{e_h}_{\XR}=\nu^{\frac12}\norm{\ur_1-\ur_2}_{1,h}$ 
and $\norm{w}_{\XR}=\nu^{\frac12}\norm{w}_{1,h}$, dividing by $\nu^{\frac12}$ proves
\eqref{eq:Sr_Lip_all}. 
\end{proof}

\begin{theorem}[Existence, uniqueness, and Picard convergence via Banach]
\label{thm:NS_existence_Banach}
Fix $\bar\gamma\in(0,1)$ and consider the set $\IW_{\bar\gamma}$ and the reduced
Oseen solution map $S_r$ from \cref{thm:NS_existence_Brouwer}.
Assume the hypotheses of \cref{thm:NS_existence_Brouwer,lem:Sr_Lip_Banach}.
Assume moreover the data smallness condition
\begin{subequations}
  \begin{equation}\label{eq:Banach_small_data}
    \banach\,\frac{\norm{f}_{\Th}}{\nu^{2}}\le \frac14,
  \end{equation}
and assume that the mesh is fine enough so that the pointwise resolution condition
  \begin{equation}\label{eq:Banach_resolution}
    \Cres{w_h}\,C_\gamma\,h^{1-\frac d4}\,
    \frac{\norm{f}_{\Th}}{\nu^{2}}\le \frac14
  \end{equation}
holds for every $w_h\in\IW_{\bar\gamma}$ satisfying $ \frac{\norm{w_h}_{1,h}}{\nu} \le \stabsolr{*}\,\frac{\norm{f}_{\Th}}{\nu^2}. $
\end{subequations}
Here $\banach$ and the pointwise constant $\Cres{w_h}$ are the constants from \eqref{eq:Sr_Lip_all}.
Then $S_r:\IW_{\bar\gamma}\to \IW_{\bar\gamma}$ has a unique fixed point
$u_h\in \IW_{\bar\gamma}$
which 
solves the reduced Navier--Stokes problem, i.e.\
\eqref{eq:redtdgoseen} with convection field $w_h=u_h$.
Moreover, by \cref{cor:uniquesol}, there exists $p_h\in \IP^{k-1}(\Th)/\IR$
such that $(u_h,p_h)\in\IX$ solves the full embedded Trefftz--DG
Navier--Stokes formulation \eqref{eq:tdgns}.
For any initial guess $w_h^0\in \IW_{\bar\gamma}$, the Picard iterates
$
w_h^{n+1}:=S_r(w_h^n)
$
converge in $\norm{\cdot}_{1,h}$ to $u_h$.
\end{theorem}
\begin{proof}
Set
\[
M:=\stabsolr{*}\,\frac{\norm{f}_{\Th}}{\nu^2},
\qquad
\IB_{\bar\gamma,M}
:=
\left\{w_h\in\IW_{\bar\gamma}:\frac{\norm{w_h}_{1,h}}{\nu}\le M\right\}.
\]
By \cref{thm:NS_existence_Brouwer}, $S_r(\IW_{\bar\gamma})\subset \IW_{\bar\gamma}$.
Moreover, the energy estimate \eqref{eq:oseen_energy} gives, for every $w_h\in\IW_{\bar\gamma}$,
\[
\frac{\norm{S_r(w_h)}_{1,h}}{\nu}
=\nu^{-\frac12}\norm{S_r(w_h)}_{\XR}
\le
\stabsolr{*}\,\frac{\norm{f}_{\Th}}{\nu^2}
=M.
\]
Thus $S_r(\IW_{\bar\gamma})\subset\IB_{\bar\gamma,M}$ and, in particular,
$S_r(\IB_{\bar\gamma,M})\subset\IB_{\bar\gamma,M}$.
Let $w_1,w_2\in \IB_{\bar\gamma,M}$. By \eqref{eq:Sr_Lip} and $\nu|w|_{h,d}\le C_\gamma h^{1-\frac d4}\vdgnorm{w}$, we obtain
\[
\norm{S_r(w_1)-S_r(w_2)}_{1,h}
\le
\frac{\norm{f}_{\Th}}{\nu^{2}}
\left(
\banach
+
\Cres{w_1}\,C_\gamma\,h^{1-\frac d4}
\right)\norm{w_1-w_2}_{1,h}.
\]
Since $w_1\in\IB_{\bar\gamma,M}$, condition \eqref{eq:Banach_resolution} applies to $w_1$.
Together with \eqref{eq:Banach_small_data}, this shows that the whole Lipschitz prefactor is at most $1/2$. Hence $S_r$ is a contraction on $\IB_{\bar\gamma,M}$ with respect to $\norm{\cdot}_{1,h}$.
This set is closed in the finite-dimensional polynomial space $\IW$ and therefore complete.
Banach's fixed-point theorem, cf. \cite[\S1.1]{Zeidler1986FixedPoint}, therefore yields a unique $u_h\in \IB_{\bar\gamma,M}$ such that
\[
u_h=S_r(u_h).
\]
Every fixed point in $\IW_{\bar\gamma}$ belongs to $\IB_{\bar\gamma,M}$, since any fixed point satisfies $w_h=S_r(w_h)$.
Hence the fixed point is unique in $\IW_{\bar\gamma}$.
For any initial guess $w_h^0\in\IW_{\bar\gamma}$, the iterate $w_h^1=S_r(w_h^0)$ lies in $\IB_{\bar\gamma,M}$, and the remaining Picard iterates converge to $u_h$ in $\norm{\cdot}_{1,h}$.
By definition of $S_r$, the fixed point $u_h\in\XR$ solves the reduced Navier--Stokes problem.
Finally, \cref{cor:uniquesol} yields a pressure $p_h\in \IP^{k-1}(\Th)/\IR$ such that $(u_h,p_h)\in\IX$ solves the full formulation \eqref{eq:tdgns}.
\end{proof}

\newcommand{\DG}{\mathrm{DG}}
\newcommand{\T}{\mathrm{T}}
\newcommand{\loc}{\mathrm{loc}}

\subsection{Convergence analysis}\label{sec:ns_convergence}
To show convergence of the embedded Trefftz-DG solution to the exact Navier--Stokes solution, we will transfer the error control from the underlying DG problem to the Trefftz-DG problem.
In contrast to the existence and uniqueness analysis above, we do \emph{not} work on the reduced formulation here, since the reduced spaces associated with the standard DG scheme and the embedded Trefftz--DG scheme are not naturally the same.
Instead, we compare both methods on the common space $\IX$.

Accordingly, throughout this section we assume from the outset that both nonlinear fixed-point problems are well posed in their respective Banach regimes: the DG Navier-Stokes discretization admits a unique fixed point, and so does the embedded Trefftz-DG discretization.
For a fixed convecting field $w_h\in \IW_{\bar\gamma}$, both methods satisfy the same global Oseen equation; their difference is therefore entirely driven by the local equation enforcing the Trefftz constraint.
This allows us to transfer error control from the local defect of the DG solution to the difference of the nonlinear fixed points.

\begin{corollary}[Uniform lower bound for the full coupled Trefftz inf-sup constant in the Banach regime]
\label{cor:trefftz_full_infsup_uniform}
Assume the hypotheses of \cref{thm:NS_existence_Banach}, and let
$w_h\in\IW_{\bar\gamma}$ be the fixed point provided by
\cref{thm:NS_existence_Banach}. Then there exists a constants
$\stabtrefftz{*}>0$, $\contoseen{*,\bar\gamma}<\infty$,
depending only on $k$, the DG parameters, the shape-regularity, $\bar\gamma$, and the Banach regime constants, such that
\begin{equation}\label{eq:trefftz_full_infsup_uniform}
    \conttrefftz{w_h}\lesssim \conttrefftz{*}
    \qquad \text{and}\qquad
    \stabtrefftz{w_h}\gtrsim \stabtrefftz{*}>0 .
\end{equation}
\end{corollary}
\begin{proof}
Since $w_h=S_r(w_h)$, the reduced Oseen stability estimate
\eqref{eq:oseen_energy} gives
\[
\frac{\norm{w_h}_{1,h}}{\nu}
=
\nu^{-1/2}\|w_h\|_{\XR}
\le
\stabsolr{*}\,\frac{\norm{f}_{\Th}}{\nu^2}.
\]
By the Banach smallness assumption \eqref{eq:Banach_small_data}, the ratio
$\norm{f}_{\Th}/\nu^2$ is uniformly bounded in terms of the Banach regime.
By \cref{lem:oseen_cont} and the discussion in \cref{rem:oseen_constants}, it follows that the Oseen continuity constant is uniformly bounded,
\[
\contoseen{w_h}
\lesssim 
1+\frac{\|w_h\|_{1,h}}{\nu}
\lesssim 
\contoseen{*},
\]
Inserting the bound on $\contoseen{w_h}$ into the Part~I full coupled inf-sup formula yields
$\stabtrefftz{w_h}\ge \stabtrefftz{*}>0,$
with the stated dependencies.
\end{proof}

\begin{theorem}[Transfer from DG fixed points to Trefftz fixed points]
\label{thm:transfer_fixed_points}
Assume that the DG Navier--Stokes problem admits a solution
$ \up_h^{\mathrm{DG}}=(u_h^{\mathrm{DG}},p_h^{\mathrm{DG}})\in \IX$
with $u_h^{\mathrm{DG}}\in \IW_{\bar\gamma}$, i.e.
\begin{equation}\label{eq:ns_dg}
\blfoseen{u_h^{\mathrm{DG}}}(\up_h^{\mathrm{DG}},\vq_h)=(f,v_h)_\Th
\qquad\forall \vq_h=(v_h,q_h)\in \IX.
\end{equation}
Assume moreover that the embedded Trefftz--DG Navier--Stokes problem is well posed on
$\IW_{\bar\gamma}$, and let $ u_h^{\IT}\in \IW_{\bar\gamma} $
be the unique fixed point of the reduced Trefftz Oseen solution map
$ S_r:\IW_{\bar\gamma}\to \IW_{\bar\gamma}, $
with contraction constant $L_{\mathrm{T}}\in(0,1)$ in the norm
$\norm{\cdot}_{1,h}$.

For each $w_h\in \IW_{\bar\gamma}$, let
\[
\up_h^{\IT}(w_h)=(u_h^{\IT}(w_h),p_h^{\IT}(w_h))\in \IX,
\]
denote the full embedded Trefftz Oseen solution corresponding to the reduced
velocity solution $u_h^{\IT}(w_h)=S_r(w_h)$, cf.\ \cref{cor:uniquesol}.
Define the local defect of the DG solution by
$ \mathcal R_{h,u_h^{\mathrm{DG}}}^{\mathrm{loc}}(\up_h^{\mathrm{DG}})\in \IQ' $
through
\[
\bigl\langle
\mathcal R_{h,u_h^{\mathrm{DG}}}^{\mathrm{loc}}(\up_h^{\mathrm{DG}}),
\rs_h
\bigr\rangle_\Th
:=
\langle \oposeen{u_h^{\mathrm{DG}}}\up_h^{\mathrm{DG}},\rs_h\rangle_\Th
-
(h\nu^{-\frac12}f,r_h)_\Th
\qquad
\forall \rs_h=(r_h,s_h)\in \IQ.
\]
Then
\begin{equation}\label{eq:transfer_fixed_points}
\norm{u_h^{\IT}-u_h^{\mathrm{DG}}}_{1,h}
\le
\frac{1}{(1-L_{\mathrm{T}})\nu^{\frac12}\stabtrefftz{u_h^{\mathrm{DG}}}}
\bigl\|
\mathcal R_{h,u_h^{\mathrm{DG}}}^{\mathrm{loc}}(\up_h^{\mathrm{DG}})
\bigr\|_{\IQ'}.
\end{equation}
\end{theorem}
\begin{proof}
We first fix a convection field $w_h\in \IW_{\bar\gamma}$ and compare the two
corresponding Oseen solutions
$\up_h^{\IT}(w_h)\in \IX$ and $\up_h^{\mathrm{DG}}(w_h)\in \IX$, where
$\up_h^{\mathrm{DG}}(w_h)\in \IX$ denotes the ambient DG Oseen solution with
frozen convector $w_h$, i.e.
\[
\blfoseen{w_h}(\up_h^{\mathrm{DG}}(w_h),\vq_h)=(f,v_h)_\Th
\qquad\forall \vq_h\in \IX.
\]

Since $\up_h^{\mathrm{DG}}(w_h)$ satisfies the global Oseen equation on all of
$\IX$, and in particular on the Trefftz test space $\IT_{w_h}\subset \IX$,
subtracting the DG formulation from the Trefftz formulation \eqref{eq:tdgoseen}
gives
\[
\blftrefftz{w_h}\bigl(
\up_h^{\IT}(w_h)-\up_h^{\mathrm{DG}}(w_h),
(\vq_h,\rs_h)
\bigr)
=
-
\bigl\langle
\mathcal R_{h,w_h}^{\mathrm{loc}}\bigl(\up_h^{\mathrm{DG}}(w_h)\bigr),
\rs_h
\bigr\rangle_\Th
\]
for all $(\vq_h,\rs_h)\in \IT_{w_h}\times \IQ$.
Applying the full inf-sup estimate for $\blftrefftz{w_h}$ yields
\begin{equation}\label{eq:transfer_full}
\norm{
\up_h^{\IT}(w_h)-\up_h^{\mathrm{DG}}(w_h)
}_{\IX}
\le
\frac{1}{\stabtrefftz{w_h}}
\bigl\|
\mathcal R_{h,w_h}^{\mathrm{loc}}\bigl(\up_h^{\mathrm{DG}}(w_h)\bigr)
\bigr\|_{\IQ'},
\end{equation}
and hence
\begin{equation}\label{eq:transfer_velocity}
\norm{
u_h^{\IT}(w_h)-u_h^{\mathrm{DG}}(w_h)
}_{1,h}
\le
\frac{1}{\nu^{\frac12}\stabtrefftz{w_h}}
\bigl\|
\mathcal R_{h,w_h}^{\mathrm{loc}}\bigl(\up_h^{\mathrm{DG}}(w_h)\bigr)
\bigr\|_{\IQ'}.
\end{equation}

We now pass to the nonlinear solutions. Since $u_h^{\IT}$ is a fixed point
of $S_r$, we have
\begin{align*}
\norm{u_h^{\IT}-u_h^{\mathrm{DG}}}_{1,h}
&\le
\norm{S_r(u_h^{\IT})-S_r(u_h^{\mathrm{DG}})}_{1,h}
+
\norm{S_r(u_h^{\mathrm{DG}})-u_h^{\mathrm{DG}}}_{1,h}
\\
&\le
L_{\mathrm{T}}\norm{u_h^{\IT}-u_h^{\mathrm{DG}}}_{1,h}
+
\norm{u_h^{\IT}(u_h^{\mathrm{DG}})-u_h^{\mathrm{DG}}}_{1,h}.
\end{align*}
Applying \eqref{eq:transfer_velocity} with $w_h=u_h^{\mathrm{DG}}$ gives
\[
\norm{u_h^{\IT}(u_h^{\mathrm{DG}})-u_h^{\mathrm{DG}}}_{1,h}
\le
\frac{1}{\nu^{\frac12}\stabtrefftz{u_h^{\mathrm{DG}}}}
\bigl\|
\mathcal R_{h,u_h^{\mathrm{DG}}}^{\mathrm{loc}}(\up_h^{\mathrm{DG}})
\bigr\|_{\IQ'}.
\]
Therefore
\[
(1-L_{\mathrm{T}})
\norm{u_h^{\IT}-u_h^{\mathrm{DG}}}_{1,h}
\le
\frac{1}{\nu^{\frac12}\stabtrefftz{u_h^{\mathrm{DG}}}}
\bigl\|
\mathcal R_{h,u_h^{\mathrm{DG}}}^{\mathrm{loc}}(\up_h^{\mathrm{DG}})
\bigr\|_{\IQ'},
\]
which proves \eqref{eq:transfer_fixed_points}.
\end{proof}

\begin{lemma}[Bound on the local defect of the DG solution]
\label{lem:local_defect_DG}
Let $\up=(u,p)\in [H^2(\Omega)]^d\times H^1(\Omega)$ be the exact solution of
\eqref{eq:sns}, let $ \up_h^{\mathrm{DG}}=(u_h^{\mathrm{DG}},p_h^{\mathrm{DG}})\in \IX $
be a DG Navier--Stokes solution, and set
$ w_h:=u_h^{\mathrm{DG}}. $
Then
\begin{equation}\label{eq:local_defect_bound}
\bigl\|
\mathcal R_{h,w_h}^{\mathrm{loc}}(\up_h^{\mathrm{DG}})
\bigr\|_{\IQ'}
\lesssim
(1+|w_h|_{h,d})\norm{\up-\up_h^{\mathrm{DG}}}_{\IX,\ast}
+
\mathfrak E_h^{\mathrm{loc}}(\up;w_h)
+
h\nu^{-\frac12}\norm{\nabla u}_{L^4(\Omega)}
\norm{u-u_h^{\mathrm{DG}}}_{1,h},
\end{equation}
where
\begin{align}
\mathfrak E_h^{\mathrm{loc}}(\up;w_h)
:=
\inf_{\eta_h=(v_h,q_h)\in \IX}
\Big(
&\norm{\up-\eta_h}_{\IX,\ast}
+\nu^{\frac12}\norm{h\Delta(u-v_h)}_{\Th}
+\nu^{-\frac12}\norm{h\nabla(p-q_h)}_{\Th}
\notag\\
&+\nu^{\frac12}\norm{\Div(u-v_h)}_{\Th}
+\nu^{\frac12}|w_h|_{h,d}\norm{\nabla(u-v_h)}_{\Th}
\Big).
\label{eq:local_best_approx}
\end{align}
\end{lemma}
\begin{proof}
Let $\eta_h=(v_h,q_h)\in \IX$ be arbitrary. Since the exact solution $\up=(u,p)$
of \eqref{eq:sns} satisfies
\[
-\nu\Delta u+(u\cdot\nabla)u+\nabla p=f,
\qquad
\Div u=0,
\]
we have
\[
\langle \oposeen{u}\up,\rs_h\rangle_\Th
=
(h\nu^{-\frac12}f,r_h)_\Th
\qquad \forall \rs_h=(r_h,s_h)\in \IQ.
\]
Hence
\begin{align*}
\mathcal R_{h,w_h}^{\mathrm{loc}}(\up_h^{\mathrm{DG}})
&=
\oposeen{w_h}\up_h^{\mathrm{DG}}-(h\nu^{-\frac12}f,0)
\\
&=
\oposeen{w_h}(\up_h^{\mathrm{DG}}-\eta_h)
+
\oposeen{w_h}(\eta_h-\up)
+
(\oposeen{w_h}-\oposeen{u})\up .
\end{align*}

For the first term, by continuity of the local operator we obtain
\begin{align*}
\bigl\|\oposeen{w_h}(\up_h^{\mathrm{DG}}-\eta_h)\bigr\|_{\IQ'}
&\le
\contlocal{w_h}\norm{\up_h^{\mathrm{DG}}-\eta_h}_{\IX,\ast}
\\
&\le
\contlocal{w_h}
\bigl(
\norm{\up-\up_h^{\mathrm{DG}}}_{\IX,\ast}
+
\norm{\up-\eta_h}_{\IX,\ast}
\bigr).
\end{align*}
Since $\contlocal{w_h}\lesssim 1+|w_h|_{h,d}$, this gives
\begin{equation}\label{eq:defect_term_1}
\bigl\|\oposeen{w_h}(\up_h^{\mathrm{DG}}-\eta_h)\bigr\|_{\IQ'}
\lesssim
(1+|w_h|_{h,d})
\bigl(
\norm{\up-\up_h^{\mathrm{DG}}}_{\IX,\ast}
+
\norm{\up-\eta_h}_{\IX,\ast}
\bigr).
\end{equation}

For the second term, by definition of \(\oposeen{w_h}\), stability of the \(L^2\)-projection, and \eqref{eq:local_conv_proj},
\begin{align}
\bigl\|\oposeen{w_h}(\eta_h-\up)\bigr\|_{\IQ'}
\lesssim\;
&\nu^{\frac12}\norm{h\Delta(v_h-u)}_{\Th}
+\nu^{-\frac12}\norm{h\nabla(q_h-p)}_{\Th}
\notag\\
&+\nu^{\frac12}\norm{\Div(v_h-u)}_{\Th}
+\nu^{\frac12}|w_h|_{h,d}\norm{\nabla(v_h-u)}_{\Th}.
\label{eq:defect_term_2}
\end{align}

For the third term, only the convection part differs, and therefore
\[
(\oposeen{w_h}-\oposeen{u})\up
=
( h\nu^{-\frac12}\Pi^{k-2}\bigl((w_h-u)\cdot\nabla u\bigr), 0 ).
\]
Thus, using the dual $L^4$ estimate and H\"older's inequality,
\begin{align*}
\bigl\|(\oposeen{w_h}-\oposeen{u})\up\bigr\|_{\IQ'}
&\lesssim
\nu^{-\frac12}
\norm{h (w_h-u)\cdot\nabla u}_{L^{\frac43}(\Omega)}
\le
h\nu^{-\frac12}
\norm{\nabla u}_{L^4(\Omega)}
\norm{w_h-u}_{L^4(\Omega)}.
\end{align*}
Applying the broken $L^4$ embedding to $w_h-u=u_h^{\mathrm{DG}}-u$ gives
\begin{equation}\label{eq:defect_term_3}
\bigl\|(\oposeen{w_h}-\oposeen{u})\up\bigr\|_{\IQ'}
\lesssim
h\nu^{-\frac12}\norm{\nabla u}_{L^4(\Omega)}
\norm{u-u_h^{\mathrm{DG}}}_{1,h}.
\end{equation}

Combining \eqref{eq:defect_term_1}, \eqref{eq:defect_term_2}, and
\eqref{eq:defect_term_3}, and then taking the infimum over
$\eta_h=(v_h,q_h)\in \IX$, yields \eqref{eq:local_defect_bound}.
If $w_h\in \IW_{\bar\gamma}$, then $|w_h|_{h,d}\le \bar\gamma$, so the hidden
constant is uniform in $w_h$.
\end{proof}

\begin{corollary}[Trefftz--DG inherits the DG convergence behaviour]
\label{cor:TDG_inherits_DG}
Assume the Banach regime of \cref{thm:NS_existence_Banach}, and let
$\up=(u,p)$ be the exact Navier--Stokes solution.
Let $ \up_h^{\mathrm{DG}}=(u_h^{\mathrm{DG}},p_h^{\mathrm{DG}})\in \IX $
be a DG fixed point satisfying
\[
\blfoseen{u_h^\DG}(\up_h^\DG,\vq_h)=(f,v_h)_\Th
\qquad\forall \vq_h=(v_h,q_h)\in\IX,
\]
and assume, for some $M_\DG<\infty$ independent of $h$, that
\[
u_h^\DG\in\IW_{\bar\gamma},
\qquad
\frac{\|u_h^\DG\|_{1,h}}{\nu}\le M_\DG .
\]
Let $u_h^{\IT}\in \IW_{\bar\gamma}$ be the unique Trefftz fixed point.

Then there exists a constant $C>0$, depending only on
$k$, the DG parameters, the shape-regularity, $\bar\gamma$,
$\stabtrefftz{*}$, $M_\DG$, and the Banach constants, but independent of $h$,
such that
\begin{align}
\norm{u-u_h^{\IT}}_{1,h}
\le\;&
\norm{u-u_h^{\mathrm{DG}}}_{1,h}
\notag\\
&\quad
+
C\Bigl(
\norm{\up-\up_h^{\mathrm{DG}}}_{\IX,\ast}
+
\mathfrak E_h^{\mathrm{loc}}(\up;u_h^{\mathrm{DG}})
+
h\nu^{-\frac12}\norm{\nabla u}_{L^4(\Omega)}
\norm{u-u_h^{\mathrm{DG}}}_{1,h}
\Bigr).
\label{eq:TDG_inherits_DG}
\end{align}
In particular, whenever the local defect is of the same order as the ambient DG
error, the Trefftz--DG velocity converges with the same rate as the DG velocity.
\end{corollary}
\begin{proof}
By the triangle inequality and \cref{thm:transfer_fixed_points},
\begin{align*}
\norm{u-u_h^{\IT}}_{1,h}
&\le
\norm{u-u_h^{\mathrm{DG}}}_{1,h}
+
\norm{u_h^{\IT}-u_h^{\mathrm{DG}}}_{1,h}
\\
&\le
\norm{u-u_h^{\mathrm{DG}}}_{1,h}
+
\frac{1}{(1-L_{\mathrm{T}})\nu^{\frac12}\stabtrefftz{u_h^{\mathrm{DG}}}}
\bigl\|
\mathcal R_{h,u_h^{\mathrm{DG}}}^{\mathrm{loc}}(\up_h^{\mathrm{DG}})
\bigr\|_{\IQ'}.
\end{align*}
Using the assumptions on $u_h^{\DG}$ we can find a uniform lower bound for $\stabtrefftz{u_h^{\mathrm{DG}}}$, depending on $\bar\gamma$ and $M_\DG$.
Applying \cref{lem:local_defect_DG} with $w_h=u_h^{\mathrm{DG}}$ yields
\eqref{eq:TDG_inherits_DG}.
\end{proof}

This result allows one to transfer convergence results from any DG method
of the form \eqref{eq:ns_dg}, provided its convection trilinear form satisfies
\cref{ass:ch}.
Two such results are available in the literature:
\begin{itemize}
    \item A detailed analysis of the global DG discretization for the Navier--Stokes problem is given in \cite{CKS05} a trilinearform fitting our assumptions.
        Under a smallness assumption on the data, the authors prove optimal convergence rates. For simplicity, the presentation in \cite{CKS05} focuses on the case $d=2$, but the analysis extends to $d=3$ as well.
    \item In \cite{DE10,DiPietroErn}, two different, suitable, trilinearforms are established and convergence of the corresponding DG approximations is established.
\end{itemize}

A pressure estimate is not pursued further here.
Once the piecewise-constant pressure part is controlled uniformly, the same transfer argument shows that the full Trefftz-DG pressure inherits the convergence behaviour of the DG pressure in the discrete norm $\norm{\cdot}_{0,h}$.

\section{Numerical results}\label{sec:numerics}

In the numerical examples we choose the discrete trilinear form $\blfconv{w}$ as a discrete counterpart of Temam's modification given by
\begin{align*}
\blfconv{w}(u,v)
:= (w\cdot\nabla u,  v)_\Omega
- \big(\avg{w}\cdot n_F \jmp{u}, \avg{v}\big)_\Fhi 
+ \frac12\big((\nabla\cdot w)u,v\big)_\Omega
-\frac12\big(\jmp{w}\cdot n_F, \avg{u\cdot v}\big)_\Fh .
\end{align*}
The choice is further discussed in \cite{DE10,DiPietroErn} and shown to satisfy our assumptions in Part~I, \cite{SVLL1_ARXIV_2026}.
For the implementation of the methods we are using \texttt{NGSolve} \cite{ngsolve} and \texttt{NGSTrefftz} \cite{ngstrefftz}.
Replication data are available in \cite{stocker_2026_20490547}.

\subsection{Kovasznay flow}\label{sec:kovasznay}
We consider the exact solution known as Kovasznay flow, which is given by
\begin{eqs}
    u = \begin{pmatrix}1-\exp(\kappa x) \cos(2\pi y)\\ \kappa/(2\pi)\exp(\kappa x) \sin(2\pi y)\end{pmatrix}, \quad
    p = -0.5\exp(2\kappa x) + \bar{p}
\end{eqs}
with $\kappa = (2\nu)^{-1}-\sqrt{(4\nu^2)^{-1}+4\pi^2}$
where $\bar{p}$ is a constant chosen such that $\int_\Omega p = 0$.
The exact solution satisfies the Navier-Stokes equations with $f = 0$ and $g = 0$.
We enforce inhomogeneous Dirichlet boundary conditions matching the exact solution.

In the experiments we choose the tolerance in the Picard iteration \cref{alg:picard} to be $\texttt{TOL} = 10^{-10}$.
If not otherwise specified the penalty parameter is chosen as $\lambda = 50k^2$.
The viscosity is set to $\nu = 1.0$.
We compare the Trefftz method with $\IT$ elements against the standard DG method with $\IP^k$ elements for $k=2,3,4,5$.

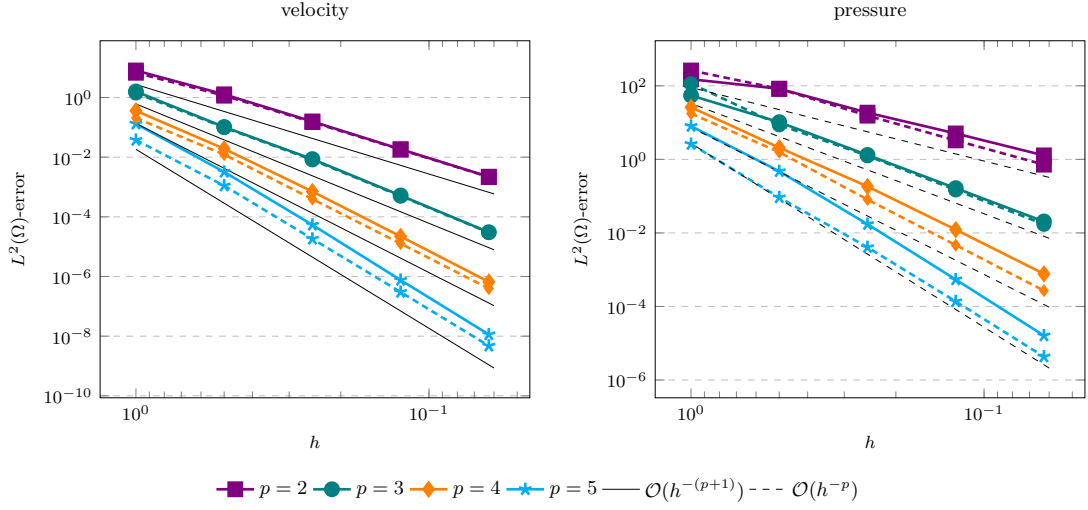
\begin{figure}[ht!]
\begin{center}    
\resizebox{.98\linewidth}{!}{
    \begin{tikzpicture}%
    \begin{groupplot}[%
      group style={%
        group name={my plots},
        group size=2 by 1,
        horizontal sep=2cm,
      },
    legend style={
        legend columns=8,
        at={(0.5,-0.2)},
        draw=none
    },
    ymajorgrids=true,
    grid style=dashed,
    cycle list name=paulcolors4,
    ]      
    \nextgroupplot[title={velocity}, ymode=log,xmode=log,x dir=reverse, ylabel={$L^2(\Omega)$-error},xlabel={$h$}]
    \foreach \p in {2,3,4,5}{
        \addplot+[discard if not={trefftz}{1},discard if not={p}{\p},discard if not={nu}{1.0}] table [x=h, y=ul2error, col sep=comma] {navst.csv};
    }
    \foreach \p in {2,3,4,5}{
        \addplot+[discard if not={trefftz}{0},discard if not={p}{\p},discard if not={nu}{1.0}] table [x=h, y=ul2error, col sep=comma] {navst.csv};
    }
    \addplot[domain=0.06:1.0] {exp(-3*ln(1/x)+1.0)};
    \addplot[domain=0.06:1.0] {exp(-4*ln(1/x)-0.5)};
    \addplot[domain=0.06:1.0] {exp(-5*ln(1/x)-2.0)};
    \addplot[domain=0.06:1.0] {exp(-6*ln(1/x)-4.0)};

    \nextgroupplot[title={pressure},ymode=log,xmode=log,x dir=reverse, ylabel={$L^2(\Omega)$-error},xlabel={$h$}]
    \foreach \p in {2,3,4,5}{
        \addplot+[discard if not={trefftz}{1},discard if not={p}{\p},discard if not={nu}{1.0}] table [x=h, y=pl2error, col sep=comma] {navst.csv};
    }
	\addlegendimage{solid}
	\addlegendimage{dashed}
    \foreach \p in {2,3,4,5}{
        \addplot+[discard if not={trefftz}{0},discard if not={p}{\p},discard if not={nu}{1.0}] table [x=h, y=pl2error, col sep=comma] {navst.csv};
    }
    \addplot[dashed,domain=0.06:1.0] {exp(-2*ln(1/x)+4.5)};
    \addplot[dashed,domain=0.06:1.0] {exp(-3*ln(1/x)+3.5)};
    \addplot[dashed,domain=0.06:1.0] {exp(-4*ln(1/x)+2.0)};
    \addplot[dashed,domain=0.06:1.0] {exp(-5*ln(1/x)+1.0)};
    \legend{$p=2$,$p=3$,$p=4$,$p=5$,$\mathcal O(h^{-(p+1)})$,$\mathcal O(h^{-p})$}
    \end{groupplot}
\end{tikzpicture}}
\end{center} 
\caption{
    Convergence of the velocity and pressure in the $L^2$-norm for the Kovasznay flow problem.
    The results for the Trefftz method are shown in solid lines, while the results for the standard finite element method are shown in dashed lines. 
    The expected convergence rates are shown as solid and dashed lines black lines.
}
\label{fig:ex1}
\end{figure}

Results for the convergence of the velocity and pressure in the $L^2$-norm are shown in \cref{fig:ex1}.
The results for the Trefftz method are shown in solid lines, while the results for the standard finite element method are shown in dashed lines.
Both methods show the expected convergence rates.

\begin{table}[ht!]\centering
{\setlength{\tabcolsep}{3pt}\begin{filecontents*}{navst.csv}
hnr,nu,h,p,alpha,trefftz,ul2error,ul2eoc,pl2error,pl2eoc,iters,totaltime
0,1.0,1.0,2,50,1,8.022866516018006,0,150.98203912865355,0,17,0.3687002658843994
1,1.0,0.5,2,50,1,1.2718612613080502,2.657176488563732,83.30023076100545,0.8579845389876142,12,0.4603145122528076
2,1.0,0.25,2,50,1,0.1571631604209591,3.016606315702665,18.620116231150032,2.161458413630056,12,1.3488433361053467
3,1.0,0.125,2,50,1,0.018469420912427,3.089052546514241,5.062378983069438,1.8789746579993165,11,4.629936218261719
4,1.0,0.0625,2,50,1,0.0022098394233573027,3.063125186999198,1.2794272437276628,1.9843174064502482,11,17.714234113693237
0,1.0,1.0,3,50,1,1.5908549831864496,0,54.737474600777155,0,12,0.29699039459228516
1,1.0,0.5,3,50,1,0.10266164913567621,3.95383308447757,10.260915704097206,2.4153693896354773,12,1.1339812278747559
2,1.0,0.25,3,50,1,0.00859426256724237,3.5783796765587135,1.3152693250298508,2.963729332989417,11,2.155832052230835
3,1.0,0.125,3,50,1,0.0005187794395839304,4.050180645690376,0.16482417159912466,2.9963585119500746,11,7.346860408782959
4,1.0,0.0625,3,50,1,3.0494402590563456e-05,4.088504945659761,0.02047625716680936,3.0089038945774944,11,32.7329843044281
0,1.0,1.0,4,50,1,0.3584461301040584,0,26.117506585787496,0,13,0.4284029006958008
1,1.0,0.5,4,50,1,0.019112638832023233,4.229157430187789,2.0775883963681756,3.652035404092977,12,0.9954118728637695
2,1.0,0.25,4,50,1,0.0007047265401529284,4.761319623671085,0.1834374751309046,3.501549553745472,11,3.1761412620544434
3,1.0,0.125,4,50,1,2.2056281565147334e-05,4.997802052513274,0.012259301896550832,3.903339670127635,11,11.844591617584229
4,1.0,0.0625,4,50,1,6.61439574800075e-07,5.059436412994442,0.000769698334098505,3.993439892365798,11,57.9602255821228
0,1.0,1.0,5,50,1,0.1314006989492496,0,8.01867524155042,0,13,1.01104736328125
1,1.0,0.5,5,50,1,0.0032022798509269506,5.358729747321981,0.4731536090955126,4.082983374973067,11,1.6639535427093506
2,1.0,0.25,5,50,1,5.3771130060405626e-05,5.896123790139726,0.01725890448398846,4.7768958329934295,11,4.836868047714233
3,1.0,0.125,5,50,1,7.604655551238688e-07,6.143805078217865,0.0005399379858364726,4.998403364383017,11,17.904821634292603
4,1.0,0.0625,5,50,1,1.1366721068857385e-08,6.063994854806908,1.609274708766608e-05,5.068311191921301,11,89.28086042404175
0,1.0,1.0,2,50,0,6.747173441730981,0,260.96413566953646,0,16,0.31040048599243164
1,1.0,0.5,2,50,0,1.1513975725087198,2.5508971722584195,82.82564428441609,1.6557021247727,11,0.49227380752563477
2,1.0,0.25,2,50,0,0.1491422616634568,2.9486250461146324,15.45084813669332,2.4223914872559367,12,1.855407953262329
3,1.0,0.125,2,50,0,0.01785648102520549,3.062169422683102,3.286342272500177,2.2331313836418754,11,6.35813045501709
4,1.0,0.0625,2,50,0,0.002162740306145241,3.045515449384455,0.7187580924427414,2.192904545438517,11,28.626416444778442
0,1.0,1.0,3,50,0,1.4122040734205588,0,114.67010978727414,0,11,0.33110976219177246
1,1.0,0.5,3,50,0,0.09878872696118188,3.837458351513051,8.69023370305696,3.721950598049168,12,0.9182047843933105
2,1.0,0.25,3,50,0,0.008108187457863325,3.6068950727134554,1.2499953906463526,2.797472200430756,11,3.646803617477417
3,1.0,0.125,3,50,0,0.0004967235390908158,4.028864422715711,0.1472012114958701,3.086061324734583,11,19.20455813407898
4,1.0,0.0625,3,50,0,2.9888136525744573e-05,4.054798260499617,0.017096968056537626,3.105977136976708,11,79.9438316822052
0,1.0,1.0,4,50,0,0.2052181783471458,0,17.59730921876774,0,12,0.5340712070465088
1,1.0,0.5,4,50,0,0.012164671242520138,4.076389295158318,1.5693088401239121,3.4871536372630243,11,1.8476898670196533
2,1.0,0.25,4,50,0,0.00039323524393419714,4.951160891042408,0.08088647581302047,4.278086987764265,11,8.98195743560791
3,1.0,0.125,4,50,0,1.2685690189063897e-05,4.954118710630338,0.004662191867966389,4.116818221264007,11,42.07612705230713
4,1.0,0.0625,4,50,0,3.9470769956477175e-07,5.0062735417113915,0.000269059051569233,4.1150136311484395,11,183.18919277191162
0,1.0,1.0,5,50,0,0.03747335640028913,0,2.5958937959270036,0,12,0.8709478378295898
1,1.0,0.5,5,50,0,0.0011090003611157254,5.07853346186515,0.09227908219442046,4.814083895217005,11,3.670307159423828
2,1.0,0.25,5,50,0,1.8247748930353386e-05,5.925397523101133,0.004071620931279646,4.502328497498988,11,19.8352472782135
3,1.0,0.125,5,50,0,3.012044230168945e-07,5.920831736394782,0.0001367944345323486,4.895521811840231,11,83.69158148765564
4,1.0,0.0625,5,50,0,4.747695268609781e-09,5.987371806705032,4.307410131452969e-06,4.989045028495349,11,373.2896592617035
\end{filecontents*}
\pgfplotstabletypeset[
col sep=comma,
fixed,fixed zerofill=true,
every head row/.style={before row=\toprule,after row=\midrule},
every last row/.style={after row=\bottomrule},
create on use/nhnr/.style={create col/set list={1.0,$2^{-1}$,$2^{-2}$,$2^{-3}$,$2^{-4}$}},
columns={nhnr,totaltime,iters,totaltime,iters,totaltime,iters,totaltime,iters,totaltime,iters,totaltime,iters},
display columns/0/.style={column name={Meshsize},string type,column type=l,discard if not={trefftz}{1},discard if not={p}{2},discard if not={nu}{1.0}},
display columns/1/.style={column name={$\IT^3$},discard if not={trefftz}{1},discard if not={p}{3},discard if not={nu}{1.0},column type/.add={r@{\hspace{2pt}}}{}},
display columns/2/.style={column name={},fixed,precision=0,fixed zerofill=false,discard if not={trefftz}{1},discard if not={p}{3},discard if not={nu}{1.0},postproc cell content/.append code={\pgfkeysalso{@cell content/.add={(}{)}}},column type/.add={r@{\hspace{10pt}}}{}},
display columns/3/.style={column name={$\IP^3$},discard if not={trefftz}{0},discard if not={p}{3},discard if not={nu}{1.0},column type/.add={r@{\hspace{2pt}}}{}},
display columns/4/.style={column name={},fixed,precision=0,fixed zerofill=false,discard if not={trefftz}{0},discard if not={p}{3},discard if not={nu}{1.0},postproc cell content/.append code={\pgfkeysalso{@cell content/.add={(}{)}}},column type/.add={r@{\hspace{10pt}}}{}},
display columns/5/.style={column name={$\IT^4$},discard if not={trefftz}{1},discard if not={p}{4},discard if not={nu}{1.0},column type/.add={r@{\hspace{2pt}}}{}},
display columns/6/.style={column name={},fixed,precision=0,fixed zerofill=false,discard if not={trefftz}{1},discard if not={p}{4},discard if not={nu}{1.0},postproc cell content/.append code={\pgfkeysalso{@cell content/.add={(}{)}}},column type/.add={r@{\hspace{10pt}}}{}},
display columns/7/.style={column name={$\IP^4$},discard if not={trefftz}{0},discard if not={p}{4},discard if not={nu}{1.0},column type/.add={r@{\hspace{2pt}}}{}},
display columns/8/.style={column name={},fixed,precision=0,fixed zerofill=false,discard if not={trefftz}{0},discard if not={p}{4},discard if not={nu}{1.0},postproc cell content/.append code={\pgfkeysalso{@cell content/.add={(}{)}}},column type/.add={r@{\hspace{10pt}}}{}},
display columns/9/.style={column name={$\IT^5$},discard if not={trefftz}{1},discard if not={p}{5},discard if not={nu}{1.0},column type/.add={r@{\hspace{2pt}}}{}},
display columns/10/.style={column name={},fixed,precision=0,fixed zerofill=false,discard if not={trefftz}{1},discard if not={p}{5},discard if not={nu}{1.0},postproc cell content/.append code={\pgfkeysalso{@cell content/.add={(}{)}}},column type/.add={r@{\hspace{10pt}}}{}},
display columns/11/.style={column name={$\IP^5$},discard if not={trefftz}{0},discard if not={p}{5},discard if not={nu}{1.0},column type/.add={r@{\hspace{2pt}}}{}},
display columns/12/.style={column name={},fixed,precision=0,fixed zerofill=false,discard if not={trefftz}{0},discard if not={p}{5},discard if not={nu}{1.0},postproc cell content/.append code={\pgfkeysalso{@cell content/.add={(}{)}}},column type=r},
]{navst.csv}}
\caption{Total timings in seconds and iteration count for the Kovasznay flow problem on various mesh sizes and polynomial degrees.}\label{tab:ex1}
\end{table}

In \cref{tab:ex1} we show the total timings in seconds and the number of Picard iterations needed to reach the tolerance $\texttt{TOL} = 10^{-10}$.
Already on quite coarse meshes the Trefftz method is faster than the standard finite element method, especially for higher polynomial degrees.
The number of Picard iterations is comparable for both methods.

\subsection{Sensitivity to the Reynolds number}\label{sec:reynolds}
We consider the same problem as in \cref{sec:kovasznay}, but with varying viscosity $\nu$.
The goal is to investigate the sensitivity of the methods to the Reynolds number, and the dependence on the penalty parameter $\lambda$ and mesh size $h$.

In \cref{fig:reynolds} we consider $\nu=0.25^i$ for $i=0,\dots,8$, and $\lambda=2^j$ for $j=3,\dots,8$ on three different mesh sizes $h=0.5,0.25,0.125$.
We fix the polynomial degree to $p=4$ and compare the Trefftz method with the standard DG method in terms of the number of Picard iterations needed to reach the tolerance $\texttt{TOL} = 10^{-8}$.
If the method fails to converge within 100 iterations, we mark the point with a cross.

\begin{figure}[!htb]
	\resizebox{\linewidth}{!}{
		\begin{tikzpicture}
			\begin{groupplot}[%
				group style={%
					group size=3 by 2,
					horizontal sep=1em,
					vertical sep=1em,
					xticklabels at=edge bottom,
					yticklabels at=edge left,
					x descriptions at=edge bottom, 
					y descriptions at=edge left,
				},
				xlabel={Viscosity $\nu$},
				ylabel={Penalty parameter $\lambda$},
				xmode=log,
				ymode=log,
				scatter,
				point meta min=0,
				point meta max=100,
                cycle list={
                    {only marks, mark=square*, mark size=5, },
                    {only marks, mark=x, mark size=6.5, very thick, black,}
                }
			]

				\nextgroupplot[]
                \addplot+[ scatter src=explicit, discard if not={trefftz}{1}, discard if not={h}{0.5}, discard if={iters}{-1},]
				table[ col sep=comma, x=nu, y=alpha, meta=iters ] {navstnuo4.csv};
                \addplot+[  discard if not={trefftz}{1}, discard if not={h}{0.5}, discard if not={iters}{-1},]
				table[ col sep=comma, x=nu, y=alpha ] {navstnuo4.csv};
				\nextgroupplot[]
                \addplot+[ scatter src=explicit, point meta=explicit, discard if not={trefftz}{1}, discard if not={h}{0.25}, discard if={iters}{-1},]
				table[ col sep=comma, x=nu, y=alpha, meta=iters ] {navstnuo4.csv};
                \addplot+[  discard if not={trefftz}{1}, discard if not={h}{0.25}, discard if not={iters}{-1},]
				table[ col sep=comma, x=nu, y=alpha ] {navstnuo4.csv};
				\nextgroupplot[]
                \addplot+[ scatter src=explicit,  point meta=explicit, discard if not={trefftz}{1}, discard if not={h}{0.125}, discard if={iters}{-1},]
				table[ col sep=comma, x=nu, y=alpha, meta=iters ] {navstnuo4.csv};
                \addplot+[  discard if not={trefftz}{1}, discard if not={h}{0.125}, discard if not={iters}{-1},]
				table[ col sep=comma, x=nu, y=alpha ] {navstnuo4.csv};

				\nextgroupplot[]
                \addplot+[ scatter src=explicit,  point meta=explicit, discard if not={trefftz}{0}, discard if not={h}{0.5}, discard if={iters}{-1},]
				table[ col sep=comma, x=nu, y=alpha, meta=iters ] {navstnuo4.csv};
                \addplot+[  discard if not={trefftz}{0}, discard if not={h}{0.5}, discard if not={iters}{-1},]
				table[ col sep=comma, x=nu, y=alpha ] {navstnuo4.csv};
				\nextgroupplot[]
                \addplot+[ scatter src=explicit,  point meta=explicit, discard if not={trefftz}{0}, discard if not={h}{0.25}, discard if={iters}{-1},]
				table[ col sep=comma, x=nu, y=alpha, meta=iters ] {navstnuo4.csv};
                \addplot+[  discard if not={trefftz}{0}, discard if not={h}{0.25}, discard if not={iters}{-1},]
				table[ col sep=comma, x=nu, y=alpha ] {navstnuo4.csv};
				\nextgroupplot[ colorbar, colorbar style={at={(1.05,1.0)},anchor=west} ]
                \addplot+[ scatter src=explicit,  point meta=explicit, discard if not={trefftz}{0}, discard if not={h}{0.125}, discard if={iters}{-1},]
				table[ col sep=comma, x=nu, y=alpha, meta=iters ] {navstnuo4.csv};
                \addplot+[  discard if not={trefftz}{0}, discard if not={h}{0.125}, discard if not={iters}{-1},]
				table[ col sep=comma, x=nu, y=alpha ] {navstnuo4.csv};

			\end{groupplot}

			\node[font=\bfseries, anchor=south] at ($(group c1r1.north)+(0,4mm)$) {Meshsize $h=0.5$};
			\node[font=\bfseries, anchor=south] at ($(group c2r1.north)+(0,4mm)$) {Meshsize $h=0.25$};
			\node[font=\bfseries, anchor=south] at ($(group c3r1.north)+(0,4mm)$) {Meshsize $h=0.125$};
			\node[font=\bfseries, rotate=90, anchor=south] at ($(group c1r1.west)+(-12mm,0)$) {Trefftz method};
			\node[font=\bfseries, rotate=90, anchor=south] at ($(group c1r2.west)+(-12mm,0)$) {Standard DG method};

		\end{tikzpicture}}
	\caption{
		    Number of Picard iterations needed to reach the tolerance $\texttt{TOL} = 10^{-8}$ for varying viscosity $\nu$ and penalty parameter $\lambda$.
            From left to right, the columns correspond to increasing mesh refinement, $h=0.5,0.25,0.125$.
            The results for the Trefftz method are shown in the top row, while the results for the standard DG method are shown in the bottom row.
	}
	\label{fig:reynolds}
\end{figure}
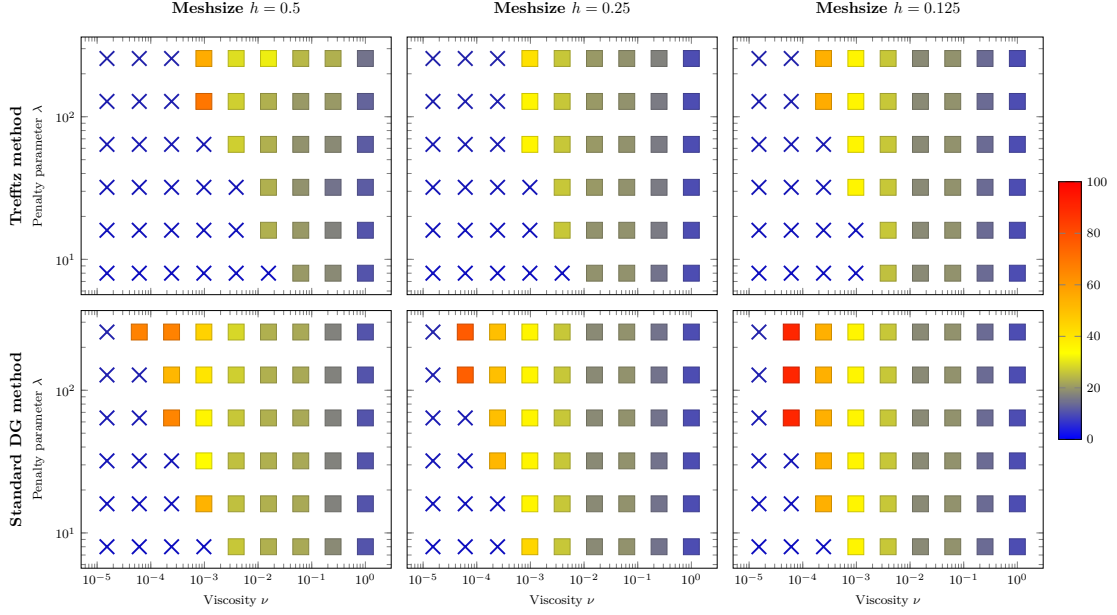

For both methods, the number of iterations increases as $\nu$ decreases, and larger values of~$\lambda$ generally improve robustness.
Moreover, refining the mesh enlarges the region in parameter space in which convergence is obtained.
The standard DG method is consistently more robust with respect to decreasing~$\nu$: on each mesh it still converges for smaller viscosities and smaller penalty parameters than the Trefftz method.
By contrast, whenever both methods converge, the observed Picard iteration counts are typically of comparable size.

\subsection{The Sch\"afer--Turek benchmark}\label{sec:schafer_turek}
We consider the classical Sch\"afer--Turek DFG benchmark 2D-1 for steady flow past a cylinder, as introduced in \cite{stbench}.
The stationary incompressible Navier--Stokes equations are posed on
\[
\Omega = [0,2.2]\times[0,0.41]\setminus B_{0.05}(0.2,0.2),
\]
with viscosity $\nu=10^{-3}$.
We impose no-slip boundary conditions on the upper and lower channel walls as well as on the cylinder boundary, and prescribe the parabolic inflow profile
\[
u(0,y)=\left(\frac{4U_{\max}\,y(0.41-y)}{0.41^2},\,0\right),
\qquad U_{\max}=0.3.
\]
The mean inflow velocity is therefore $U_{\mathrm{mean}}=\frac23 U_{\max}=0.2,$ and with cylinder diameter $L=0.1$ this yields the Reynolds number $\mathrm{Re}=\frac{U_{\mathrm{mean}}L}{\nu}=20.$

The quantities measured in \cite{stbench} are the drag and lift coefficients, the pressure difference.
Let $\eta$ denote the outer unit normal of the cylinder boundary $S$ (pointing into the fluid).
Following the benchmark convention, we define the (non-symmetric) stress tensor
$\sigma(u,p) = \nu \nabla u - p I,$
and the drag/lift force vector by
\[
\binom{F_D}{F_L} := \int_S \sigma(u,p)\,\eta \, ds.
\qquad
C_D := \frac{2}{U_{\mathrm{mean}}^2\,L}F_D,
\qquad
C_L := \frac{2}{U_{\mathrm{mean}}^2\,L}F_L,
\]
where $C_D$ and $C_L$ are the dimensionless drag and lift coefficients.
Moreover, we consider the pressure difference
\[
p_{\mathrm{diff}}:=p(a_1)-p(a_2),
\qquad
a_1=(0.15,0.2),\quad a_2=(0.25,0.2).
\]

As reference values we use the computations reported in \cite{stbench2},
\[
C_D^{\mathrm{ref}}=5.57953523384,\qquad
C_L^{\mathrm{ref}}=0.010618948146,\qquad
p_{\mathrm{diff}}^{\mathrm{ref}}=0.11752016697.
\]
In addition, the original benchmark paper gives the admissible intervals
\[
C_D\in[5.5700,5.5900],\qquad
C_L\in[0.0104,0.0110],\qquad
p_{\mathrm{diff}}\in[0.1172,0.1176].
\]

We solve the problem on the mesh with $h=0.07$ for polynomial degrees $p=3,\dots,8$, and compare the Trefftz-DG and standard DG discretizations.
The resulting values for $C_D$, $C_L$, and $p_{\mathrm{diff}}$, together with their absolute errors with respect to the reference values above, are reported in \cref{tab:st}.

\begin{table}[!ht]\centering
\begin{filecontents*}{st.csv}
hnr,h,p,CD,eCD,CL,eCL,pdiff,epdiff,trefftz
0,0.07,3,5.374162814143808,0.20537241969619213,0.02532770292507221,0.01470875477907221,0.1113544360988439,0.0061657308711561,True
0,0.07,4,5.586510867487598,0.006975633647598123,0.024418796515746615,0.013799848369746615,0.11273408475437707,0.004786082215622933,True
0,0.07,5,5.581095868353134,0.0015606345131340404,0.011743417116890205,0.001124468970890205,0.11629739234630884,0.001222774623691164,True
0,0.07,6,5.579659497232792,0.00012426339279247856,0.010025350770845702,0.0005935973751542978,0.11735069999465886,0.00016946697534114818,True
0,0.07,7,5.579641304881067,0.00010607104106696852,0.010223627046078354,0.00039532109992164617,0.11761395362533913,9.378665533912867e-05,True
0,0.07,8,5.579488300261416,4.693357858354119e-05,0.010508003818347943,0.0001109443276520565,0.11756039683057168,4.0229860571672216e-05,True
0,0.07,3,5.467909532209265,0.11162570163073493,-0.09428123952294777,0.10490018766894776,0.11818727365629088,0.000667106686290872,False
0,0.07,4,5.5400647335571165,0.039470500282883236,-0.007952852498202444,0.018571800644202444,0.11842410783458189,0.0009039408645818842,False
0,0.07,5,5.574418339994337,0.005116893845662496,0.010042850949581524,0.0005760971964184763,0.11784333925104773,0.00032317228104772333,False
0,0.07,6,5.579412230836806,0.00012300300319356694,0.010373525572106402,0.0002454225738935979,0.11759919154095744,7.90245709574311e-05,False
0,0.07,7,5.579755614031726,0.00022038019172665457,0.010527291564966009,9.165658103399074e-05,0.11753831811503973,1.815114503972426e-05,False
0,0.07,8,5.579699909835494,0.00016467599549407907,0.010611164936752777,7.78320924722277e-06,0.11752351645699391,3.34948699390325e-06,False
\end{filecontents*}
\pgfplotstabletypeset[
col sep=comma,
fixed,
fixed zerofill=true,
columns={method,p,CD,eCD,CL,eCL,pdiff,epdiff},
create on use/method/.style={
    create col/assign/.code={
        \ifnum\pgfplotstablerow=0
            \pgfkeyssetvalue{/pgfplots/table/create col/next content}{\multirow{6}{*}{TDG}}
        \else\ifnum\pgfplotstablerow=6
            \pgfkeyssetvalue{/pgfplots/table/create col/next content}{\multirow{6}{*}{DG}}
        \else
            \pgfkeyssetvalue{/pgfplots/table/create col/next content}{}
        \fi\fi
    }
},
every head row/.style={before row=\toprule,after row=\midrule},
every row no 5/.style={after row=\midrule},
every last row/.style={after row=\bottomrule},
display columns/0/.style={
    column name={method},
    string type,
    column type=l
},
display columns/1/.style={column name={$p$},fixed,precision=0,column type=r},
display columns/2/.style={column name={$C_D$},fixed,precision=6,column type=r},
display columns/3/.style={column name={$|e(C_D)|$},sci,sci zerofill,precision=2,column type=r},
display columns/4/.style={column name={$C_L$},fixed,precision=6,column type=r},
display columns/5/.style={column name={$|e(C_L)|$},sci,sci zerofill,precision=2,column type=r},
display columns/6/.style={column name={$\Delta p$},fixed,precision=6,column type=r},
display columns/7/.style={column name={$|e(\Delta p)|$},sci,sci zerofill,precision=2,column type=r}
]{st.csv}
\caption{Drag/lift coefficients and pressure difference with errors.}
\label{tab:st}
\end{table}

The results in \cref{tab:st} show that both methods capture the benchmark quantities with high accuracy.
In particular, both discretizations approach the reference values rapidly as the polynomial degree increases, and from $p=6$ onward the remaining errors are already very small.
The comparison between Trefftz-DG and standard DG is mixed but overall balanced:
Trefftz-DG tends to be slightly more accurate for the drag coefficient $C_D$,
while standard DG is consistently more accurate for the pressure difference $p_{\mathrm{diff}}$.
For the lift coefficient $C_L$, Trefftz-DG has an advantage at lower polynomial degrees, whereas standard DG is more accurate at higher orders.
Thus, for this benchmark the two discretizations perform very similarly overall.

\begin{figure}[!ht]
\centering
\includegraphics[width=0.8\textwidth, trim=0 7cm 0 7cm, clip]{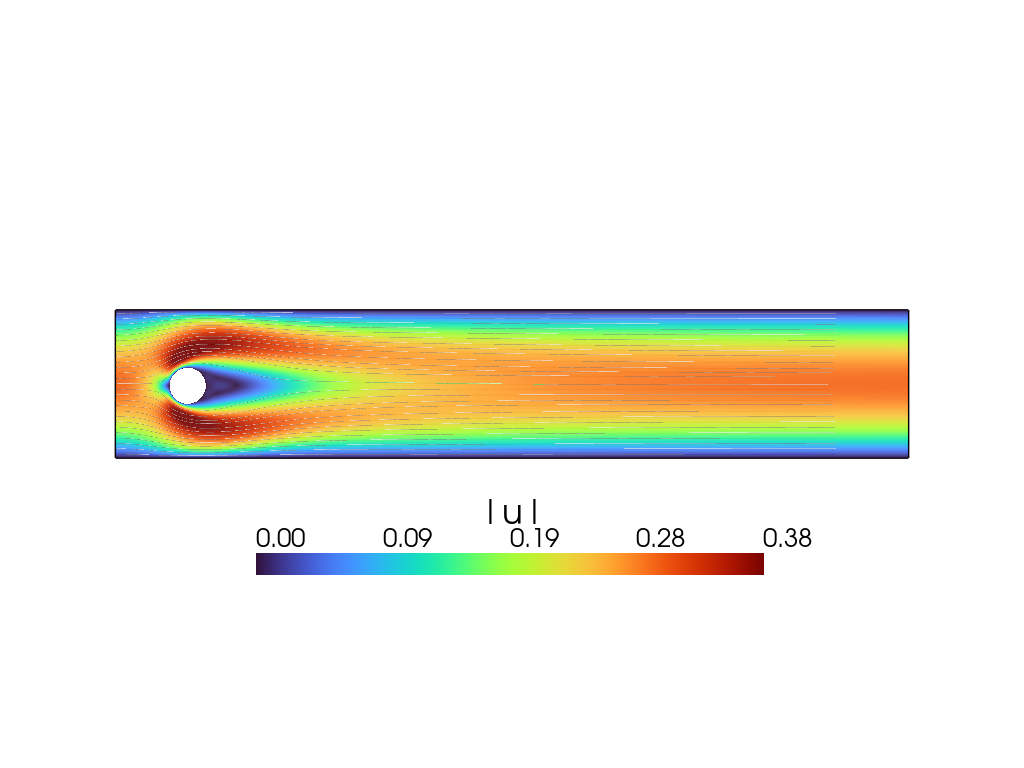}
\caption{
    Streamlines of the velocity field for the Sch\"afer--Turek benchmark problem computed with the Trefftz-DG method with polynomial degree $k=5$.
}
\label{fig:st_streamlines}
\end{figure}

A streamline plot of the computed flow is shown in \cref{fig:st_streamlines}; it exhibits the expected steady recirculation region downstream of the cylinder.

\section*{Acknowledgements}
\sloppy{
This research was funded in part by the Austrian Science Fund (FWF) \href{https://doi.org/10.55776/ESP4389824}{10.55776/ESP4389824}.
For open access purposes, the author has applied a CC BY public copyright license to any author-accepted manuscript version arising from this submission.}

\bibliographystyle{abbrvurl}
\bibliography{bib}

\end{document}